\documentclass[a4paper,12pt]{article}

\usepackage{amsthm,amsmath,stmaryrd,bbm,hyperref,geometry,color,authblk}
\usepackage[utf8]{inputenc}
\usepackage{amssymb}
\usepackage[english]{babel}
\usepackage{graphicx}
\usepackage{amsfonts,amssymb}
\usepackage{verbatim}
\usepackage{enumitem}

\usepackage{algorithm}
\usepackage[noend]{algpseudocode}
\makeatletter
\def\BState{\State\hskip-\ALG@thistlm}
\makeatother

\newtheorem{thm}{Theorem}
\newtheorem{prop}[thm]{Proposition}

\newtheorem{lem}[thm]{Lemma}

\newtheorem{hyp}{Assumption}

\newtheorem*{defi*}{Definition}
\newtheorem*{hyp*}{Assumptions}
\newtheorem{remark}[thm]{Remark}

\newcommand{\po}{\left(}
\newcommand{\pf}{\right)}
\newcommand{\co}{\left[}
\newcommand{\cf}{\right]}
\newcommand{\cco}{\llbracket}
\newcommand{\ccf}{\rrbracket}
\newcommand{\R}{\mathbb R} 
\newcommand{\T}{\mathbb T} 
\newcommand{\Z}{\mathbb Z} 
 
\newcommand{\N}{\mathbb N} 
\newcommand{\J}{\mathcal J} 
\newcommand{\dd}{\text{d}}
\newcommand{\na}{\nabla}
  
\newcommand{\1}{\mathbbm{1}}
\newcommand{\nv}[1]{#1}

\title{Adaptive force biasing algorithms: new convergence results and tensor approximations of the bias}

\author[1]{Virginie Ehrlacher}
\author[1]{Tony Lelièvre}
\author[2]{Pierre Monmarché}
\affil[1]{Université Paris-Est - CERMICS (ENPC) - INRIA}
\affil[2]{Sorbonne Universit\'e -LJLL - LCT}

\begin{document}
\maketitle

\begin{abstract}\
A modification of the Adaptive Biasing Force  method is introduced, in which the free energy is approximated by a sum of tensor products of one-dimensional functions. This  enables to handle a larger number of reaction coordinates than the classical algorithm. We prove the algorithm is well-defined and prove the long-time convergence toward a regularized version of the free energy for an idealized version of the algorithm. Numerical experiments demonstrate that the method is able to capture correlations between reaction coordinates.

\medskip

\noindent\textbf{keywords:} Monte Carlo methods ; tensor ; free energy ; importance sampling ; molecular dynamics.

\noindent\textbf{MSC class (2010):} 65C05 ; 65N12.
\end{abstract}

%

\section{Introduction}

Consider $x\in \T^D$ a vector representing the positions of particles with periodic boundary conditions ($\T = \R/\Z$), and a potential energy $V\in \mathcal C^\infty \po \T^D\pf$. 
We are interested in computing expectations of the form

\begin{eqnarray*}
\frac{1}{\int_{\T^D} e^{-\beta V(x)}\dd x}\int_{\T^D} \varphi(x) e^{-\beta V(x)}\dd x & =: & \int_{\T^D} \varphi \dd \mu_{V,\beta}
\end{eqnarray*}
where $\varphi: \T^D \to \mathbb{R}$ is called an observable and $\dd \mu_{V,\beta}:=e^{-\beta V(x)}\dd x$ is the Gibbs law with potential $V$ and inverse temperature $\beta>0$. The large dimension $D$ is so significant that, in practice, 
these quantities have to be computed with Markov Chain Monte Carlo (MCMC) algorithms, which consist in approximating the average of $\varphi$ with respect to $\mu_{V,\beta}$ along dynamics that are ergodic with respect to $\mu_{V,\beta}$. A typical sampler is the overdamped Langevin dynamics
\begin{eqnarray*}
\dd X_t  & = & -\na V(X_t)\dd t + \sqrt{2\beta^{-1}} \dd B_t
\end{eqnarray*} 
where $(B_t)_{t\geqslant 0}$ is a Brownian motion over $\mathbb T^D$. It is ergodic with invariant measure $\mu_{V,\beta}$, so that
\begin{eqnarray*}
\frac1t\int_0^t \varphi(X_s)\dd s & \underset{t\rightarrow\infty}\longrightarrow & \int_{\T^D} \varphi \dd \mu_{V,\beta}\nv{\qquad a.s.}
\end{eqnarray*}
\nv{for all measurable bounded $\varphi$, see e.g. \cite{ActaNumerica} and references therein.} Nevertheless, the convergence of the process (or, in practice, of any alternative Markov process with invariant measure $\mu_{V,\beta}$) toward its equilibrium in the long-time limit may be very slow. 
This is due to the so-called metastability phenomenon, according to which the process remains for  long times in some region of the space, with very rare transitions from one of these metastable regions to another. 
This is related to the multi-modality of the Gibbs measure and the fact MCMC algorithms typically perform local moves, so that leaving a mode of the target measure $\mu_{V,\beta}$ is a rare event. We refer to \cite{LelievreMetastable} for more details on this topic. For this reason, several adaptive methods have been developed in order to 
force the process to leave the metastable traps faster. Among those, we focus on the adaptive biasing force (ABF) algorithm, which may be seen as a particular Importance Sampling method. The general idea is to run a biased process
\begin{eqnarray}\label{eq:Xbiais}
\dd \tilde X_t  & = & -\na V(\tilde X_t)\dd t + \na V_{bias,t}(\tilde X_t)\nv{\dd t}+ \sqrt{2\beta^{-1}} \dd B_t
\end{eqnarray} 
where the biasing potential $V_{bias,t}$ is adaptively constructed from the past trajectory $(\tilde X_s)_{s\in[0,t]}$ in such a way that it is expected to converge to some $V_{bias,\infty}$. Expectations with respect to $\mu_{V,\beta}$ are then recovered through a reweighting step, assuming that ergodicity still holds:
\begin{eqnarray}\label{eq:ImportanceSampling}
\frac{\frac1t\int_0^t \varphi(\tilde X_s) e^{-\beta V_{bias,s}(\tilde X_s)}\dd s}{\frac1t\int_0^t   e^{-\beta V_{bias,s}(\tilde X_s)}\dd s} & \underset{t\rightarrow\infty}\longrightarrow & \frac{\int_{\T^D} \varphi e^{-\beta V_{bias,\infty}}  \dd \mu_{V-V_{bias,\infty},\beta}}{\int_{\T^D} e^{-\beta V_{bias,\infty}} \dd \mu_{V-V_{bias,\infty},\beta}} \ = \ \int_{\T^D} \varphi \dd \mu_{V,\beta}\,.
\end{eqnarray}
Classically, in such an Importance Sampling scheme, the aim is to design a target bias $V_{bias,\infty}$ such that two conditions are met: 1) sampling the biased equilibrium $\mu_{V-V_{bias,\infty},\beta}$ is simpler than the initial problem (i.e. the corresponding overdamped Langevin process is less metastable) and 2) the biased equilibrium is not too far from the initial target so that the exponential weights in \eqref{eq:ImportanceSampling} do not cause the asymptotical variance of the estimator to skyrocket.

In the ABF algorithm, this issue is addressed with the use of so-called reaction coordinates (or collective variables) and the associated free energy as a bias. Reaction coordinates consist of a small number $d\ll D$ of 
macroscopic coordinates of the whole microscopic system $x\in\T^D$. These coordinates are defined through a map $\xi:\T^D \rightarrow \mathcal M$ where $\mathcal M$ is a manifold of dimension $d$. 
In molecular dynamics, for example, $x\in \T^D$ is a vector which gathers the positions of all the different atoms of the system of interest, and $\xi(x)$ typically represents
some distances between particular pairs of atoms, or angles formed by some triplets of atoms. These reaction coordinates should be chosen to capture the main causes of the metastability of the system. More precisely, 
$\xi(X_t)$ should converge to equilibrium as slowly as $X_t$, while the conditional laws $\mathcal L(X\ | \ \xi(X) = z)$ for fixed $z\in \mathcal M$ when $X\sim\mu_{V,\beta}$ should be easier to sample (see \cite{LelievreABF} \nv{or Section~\ref{Subsc-discussResult}} for more detailed considerations). 
In other words, $\xi(x)$ should be a low-dimensional representation of $x$ that captures the slow variables of the system.

To these reaction coordinates $\xi$ is  associated the corresponding free energy $A : \mathcal M \to \mathbb{R}$, given by
\begin{eqnarray*}
A(z) & =& -\frac{1}{\beta} \ln \int_{\{x \in \T^D, \; \xi(x) = z\}} e^{-\beta V(x)} \delta_{\xi(x)-z}(\dd x)\,, 
\end{eqnarray*}
where $\delta_{\xi(x)-z}$ is the so-called delta measure, which can be defined from the Lebesgue measure on the submanifold $\{x\in\T^D,\ \xi(x)=z\}$ through the co-area formula, see for example \cite[Section 3.2.1]{LelievreFreeEnergy}.
This definition ensures that, if $X$ is a random variable with law $\mu_{V,\beta}$ on $\T^D$, then $\xi(X)$ is a random variable with law $\mu_{A,\beta}$ on $\mathcal M$. The heuristic of the ABF algorithm is the following. Suppose that $\mathcal M$ is compact. If we were to sample from the process 
\begin{eqnarray}\label{eq:YbiaisA}
\dd Y_t  & = & -\na \po V - A \circ \xi \pf (Y_t)\dd t + \sqrt{2\beta^{-1}} \dd B_t,
\end{eqnarray} 
the equilibrium would be $\mu_{V-A\circ \xi,\beta}$, whose image through $\xi$, by definition of $A$, is the uniform measure on $\mathcal M$. This means that there would be no more metastability along $\xi$, 
since all the regions of $\mathcal M$ would be equally visited by $\xi(Y_t)$.  Unfortunately, it is not possible to use directly this free-energy biased dynamics in practice, since it would require the knowledge of $A$ 
and thus the computation of expectations in large dimension. The idea of the ABF method is to learn $A$ on the fly, i.e. to run a process $\left(\tilde X_t\right)_{t\geqslant 0}$ solving \eqref{eq:Xbiais} with a biasing potential 
$V_{bias,t}$ constructed from $\left(\tilde X_s\right)_{s\in[0,t]}$ and designed to target $A\circ\xi$ in the longtime limit.

In practice, the choice of good reaction coordinates is a difficult problem. Up to recently, their definition has been based on the knowledge and 
intuition of specialists.  The question of the automatic learning of suitable reaction coordinates is currently a vivid research area, see for instance \cite{Schutte,Brandt} \nv{and the recent review \cite{RevueMLMD}}. 
Moreover, some techniques like the orthogonal space random walk \cite{OSRW}  provide a general way to construct new reaction coordinates from previous ones. Due to these recent progresses,
one would like to consider a relatively large  $d$. In ABF, $V_{bias,t}$ is  a function of the $d$ reaction coordinates. From a numerical point of view, since $V_{bias,t}$  is adaptively 
learned on the fly, its values have to be kept in memory, which requires a grid whose size typically scales exponentially with $d$. This limits the application of ABF to small dimensional reaction coordinates ($d\leqslant 4$). 
%
 The aim of the present work is to lift this limitation by approximating $V_{bias,t}$ using a sum of tensor products of one-dimensional functions, which reduces the size of the memory to $\mathcal O(dm)$ where $m$ is
the number of tensor terms. Remark that this can in turn help for the definition of good reaction coordinates, 
by considering as candidates a relatively large number of reaction coordinates and then conduct a statistical study to select or combine some of them. \nv{A basic idea would be to conduct a 
sensitivity analysis of the free energy, computing for instance for each reaction coordinate $\xi_i$ the best approximation in the least square  sense of the (estimated) free energy by a function only of 
the other reaction coordinates $(\xi_j)_{j\neq i}$ (which is easily done for a function given as a sum of tensor products, see \cite{konakli2016global}) and then discarding 
the reaction coordinate whose disparition gives the lowest  error.} Nevertheless, this question exceeds the scope of the present work, in which $\xi$ is supposed to be given.

\medskip

Note that the question of increasing the number of reaction coordinates in adaptive biasing algorithms has also been considered in the Bias-Exchange algorithm introduced in \cite{BiasExchange}, where several replicas of the system are run in parallel, each associated with a one-dimensional reaction coordinate. The replicas exchange their bias according to some Metropolis-Hastings probability, so that each replica eventually feels the bias in all the different directions of the reaction coordinates. Nevertheless, in this case where one-dimensional reaction coordinates are treated independently one from the others, the system remains very sensitive to correlations between reaction coordinates (the same goes for the generalized ABF introduced in \cite{ChipotEGABF}), contrary to the algorithm introduced in the present work.

Besides, let us mention that numerical methods involving both tensor approximation and Monte Carlo methods for molecular dynamics are also introduced  in  \cite{Schutte2,Schneider} for other purposes.

\medskip

In the rest of this introduction we provide  a  presentation of the ABF algorithm we consider in this work in a simple framework, and refer to   Section \ref{SubSectionDiscussion} for generalizations.  The presentation is divided into two parts. In Section~\ref{SectionRC}, we present the reference ABF algorithm we consider, without the tensor-product approximation. In Section~\ref{SubsectionTensor}, we introduce the tensor-product approximation of the bias. These two ingredients are then combined to yield the Tensor-ABF algorithm in Section~\ref{Sec:defTABF}. The two algorithms and associated convergence proofs of the reference ABF algorithm and of the tensor-product approximation are presented separately since we think they have their own interest.

\subsection{Free energy and the ABF algorithm}\label{SectionRC}

%

%
%

Let us first present the ABF algorithm in a simple framework (see \cite{ChipotHenin,DarvePorohill,LelievreABF} for more general settings). From now on, we write 
\[\mu \ =\  \mu_{V,\beta}\,,\]
 seen both as a probability law and as the density of the latter with respect to the Lebesgue measure.
 
 Let us assume that $\mathcal M=\T^d$ and that, for all $x=(q,z) \in \T^D = \T^p \times \T^d$, $\xi(x)= \xi(q,z)=z$  where $p=D-d$. 
 
 \medskip
 
At first sight, this may seem a very restrictive choice of reaction coordinates. But, using extended variables (see \cite{ChipotEABF}), this can be applied actually in very general contexts. We refer the reader to~Section~\ref{SubSectionDiscussion} 
for more details on this point.

\medskip

The associated free energy for $z\in\T^d$ is then
%
\begin{eqnarray*}
A(z)  & =& -\frac{1}{\beta}\ln \int_{\T^p} e^{-\beta V(q,z)}\dd q\,. 
\end{eqnarray*}
Following the previous discussion, our aim is then to define for all time $t\geqslant 0$ a  function $A_t$ on $\T^d$ and to sample the process
\begin{eqnarray}\label{EqEABF}
& & \left\{
\begin{array}{rcl}
\dd Q_t & = & - \na_q V (Q_t,Z_t) \dd t + \sqrt{2\beta^{-1}} \dd B^1_t\\
\dd Z_t &= &  - \na_z    V  (Q_t,Z_t)\dd t + \na_z A_t(Z_t) \dd t + \sqrt{2\beta^{-1}} \dd B^2_t\,,
\end{array}\right.
\end{eqnarray}
where $B^1$ and $B^2$ are independent Brownian motions respectively of dimension $p$ and $d$, in such a way that $A_t$ gets close to $A$ in large time. 

 Note that the free energy $A$ satisfies
\begin{eqnarray*}
\na_z A(z) & =& \frac{\int_{\T^p} \na_z V(q,z) e^{-\beta V(q,z)}\dd q}{\int_{\T^p}   e^{-\beta V(q,z)}\dd q} \ = \ \mathbb E_\mu \left[  \na_z V(Q,Z)\ |\ Z=z \right].
\end{eqnarray*} 
The following alternative equivalent characterization of $A$ will be useful in the sequel. \nv{Denoting $H^1(\T)$ the set of functions of $L^2(\T^d)$ with a weak gradient in $L^2(\T^d)$,}  define
 \begin{equation}\label{eq:defH}
 H \ := \  \left\{ f\in H^1(\T^d)\, :\, \ \int_{\T^d} f(z) \dd z = 0\right\},
 \end{equation}
and let us denote by $\mathcal{P}(\T^p \times \T^d)$ the set of probability measures on $\T^p \times \T^d = \T^D$. For all $\nu\in \mathcal{P}(\T^p \times \T^d)$ and $f\in H$, let us define
 $$
 \mathcal{E}_\nu(f):= \int_{\T^p\times\T^d} |\na_z V(q,z) - \na _z f(z)|^2 \dd \nu(q,z).
 $$
\nv{As detailed in \cite{LelievreAlrachid}}, up to an additive constant (like the potential $V$, the free energy is in fact  always defined up to an additive constant), $A$ is the unique minimizer in $H$ of the functional $\mathcal{E}_\mu$, i.e.

\begin{equation}\label{EqProblemeMini}
A \nv{-\int_{\mathbb T^d} A(z)\dd z} = \mathop{\rm argmin}_{f\in H} \mathcal{E}_\mu(f).
\end{equation}

At time $t\geqslant 0$, a trajectory $(Q_s,Z_s)_{s\in\co 0,t\cf}$ of \eqref{EqEABF} is available. Let $\nu_t$ be the probability measure on $\T^p\times \T^d = \T^D$ defined as follows: for all $\varphi \in \mathcal{C}(\T^p \times \T^d)$, 
\begin{eqnarray}\label{Eq-Empiric-Distrib}
\int_{\T^p \times \T^d}\varphi \dd \nu_t & =&  \po \int_0^t e^{-\beta A_s(Z_s)}\dd s\pf^{-1}  \int_0^t \varphi\po Q_s,Z_s\pf e^{-\beta A_s(Z_s)} \dd s\,.
\end{eqnarray}
 We call $\nu_t$ the unbiased occupation distribution of the process. By the ergodic limit \eqref{eq:ImportanceSampling}, $\nu_t$ is expected to converge \nv{weakly} to $\mu$ as $t$ goes to infinity  \nv{almost surely} (at least if $A_t$ does not change too fast with $t$). 

However, note that $\nu_t$ is a singular probability measure, so that the minimization problem 
$$
\mathop{\inf}_{f\in H} \mathcal{E}_{\nu_t}(f)
$$
is ill-posed. To circumvent this difficulty, one may consider two different alternatives to regularize the problem which we detail hereafter. Consider a smooth symmetric positive density kernel $K\in\mathcal C^\infty(\T^d\times \T^d,\R_+)$ with 
\begin{eqnarray}\label{Eq:conditionsK}
\int_{\T^d} K(y,z)\dd z = 1 & \qquad\text{and}\qquad  & K(y,z) = K(z,y)\qquad\qquad \forall y,z\in \T^d\,.
\end{eqnarray}
In practice, $K(y,\cdot)$ should be close to a Dirac mass at $y$ (see Theorem \ref{TheoremTempsLong} below). For instance, a possible choice for $K$ would be the so-called von-Mises kernel for a given small parameter $\varepsilon >0$, i.e.
\begin{equation}\label{VonMises}
K(y,z) \ \propto \  \prod_{i=1}^d \exp\po-\frac{1}{\varepsilon^2/2}\sin^2\po \frac{z_i-y_i}2 \pf\pf\,.
\end{equation}
Now, consider also a regularization parameter $\lambda\geqslant 0$. For all $\nu\in \mathcal{P}(\T^p\times \T^d)$ and all $f\in H$, we define
\begin{equation}\label{EqJ1}
 \mathcal J_{\nu} (f) \ := \ \int_{\T^p\times\T^d\times\T^d} |\na_y V(q,y) - \na_z f(z) |^2 K(y, z) \dd z\dd \nu ( q,  y)  + \lambda\int_{\T^d} |\na_z f(z)|^2 \dd z\,,
\end{equation}
Note that, as $K(y,\cdot)$ converges \nv{weakly} toward the Dirac mass at $y$ and $\lambda$ goes to 0, for all $f\in H$, $\mathcal J_\nu(f)$ converges towards $\mathcal{E}_\nu(f)$. 
The interest of introducing $\mathcal J_\nu$ is that, thanks to the regularization, the minimization problem is now well-posed:

\begin{prop}\label{prop:minimJt}
Assume that either  $K>0$ on $\T^d\times\T^d$ or $\lambda>0$. Then, for all    $\nu\in\mathcal P(\T^p\times \T^d)$, $\mathcal J_{\nu}$ admits a unique minimizer in $H$.
\end{prop}

\nv{This is  a direct consequence of the strict convexity of $\J$, see Section~\ref{SectionTensor}.} In summary, in the whole article, we work under the following conditions.

\begin{hyp}\label{Hypo}
$V\in\mathcal C^\infty(\T^D)$, $D\geqslant 3$, $\beta>0$, $\lambda \geqslant 0$ and $K\in\mathcal C^\infty(\T^d\times \T^d,\R_+)$ satisfies~\eqref{Eq:conditionsK}. Moreover, either  $K>0$ or $\lambda>0$.
\end{hyp}

\medskip

We now have all the elements to define the reference ABF algorithm in this work, see Algorithm \ref{algo:ABF} below.

\begin{algorithm}[h!]
\caption{ABF algorithm}\label{algo:ABF}
\begin{algorithmic}[1]

\BState \textbf{Input:}

\State Initial condition $(q_0,z_0) \in \T^p\times\T^d $
\State Brownian motion $(B_t^1,B_t^2)_{t\geqslant 0}$ on $\T^p\times\T^d$
\State Regularization parameters $K$, $\lambda$
\State Update period $T_{up}>0$, number of updates $N_{up}\in\N_*$, total simulation time $T_{tot} = T_{up} N_{up}$

\medskip
\BState \textbf{Output:}
\State Estimated free energy $A_{T_{tot}} \in H$
\State Trajectory $(Q_t,Z_t)_{t\in[0,T_{tot}]}\in \mathcal C\po [0,T_{tot}],\T^p\times\T^d\pf$

\medskip
\BState \textbf{Begin:}

\State Set $(Q_0,Z_0)=(q_0,z_0)$.
\State Set $A_0(z) = 0$ for all $z\in\T^d$.
\State Set $t_k = kT_{up}$ for all $k\in\cco 0,N_{up}\ccf$.

\For{$k\in\cco 1,N_{up}\ccf$}

\State Set $A_t = A_{t_{k-1}}$ for all $t\in[t_{k-1},t_k)$.
\State Set  $(Q_t,Z_t)_{t\in[t_{k-1},t_{k}]}$ to be the solution of \eqref{EqEABF} with initial condition $(Q_{t_{k-1}},Z_{t_{k-1}})$ at time $t_{k-1}$.
\State Set $A_{t_k}$ to be the minimizer in $H$ of $\mathcal J_{\nu_{t_k}}$ given by \eqref{Eq-Empiric-Distrib} and \eqref{EqJ1}.
\EndFor 
\State \textbf{end for}
 
  \medskip
\State \textbf{Return} $A_{T_{tot}}$ and $(Q_t,Z_t)_{t\in[0,T_{tot}]}$.
\end{algorithmic}
\end{algorithm}

Remark that, contrary to the cases studied in other theoretical works like \cite{LelievreABF,LelievreAlrachid,BenaimBrehier}, in Algorithm~\ref{algo:ABF}, the bias $A_t$ is piecewise constant in time, with updates at the times $t_k$, $k\in\cco 1,N_{up}\ccf$. This is due to the fact that, as will be detailed in Section~\ref{SubsectionTensor}, the bias updates are numerically demanding in our case, and thus we cannot perform them at each timestep.

We prove in Section \ref{SectionCVTempslong} the long-time convergence of Algorithm~\ref{algo:ABF}:

\begin{thm}\label{TheoremTempsLong}
Under Assumption \ref{Hypo}, let $(Q_t,Z_t,A_t)_{t\geqslant 0}$ be given by Algorithm \ref{algo:ABF}  (\nv{with a fixed $T_{up}>0$ and $N_{up}=+\infty$ so that $T_{tot}=+\infty$ and the process is defined for all positive times}). Then, as $t\rightarrow+\infty$, almost surely, $\nu_t$ \nv{given by \eqref{Eq-Empiric-Distrib}} weakly converges toward $\mu$ and
\begin{eqnarray*}
\|\na A_t - \na A_* \|_{\infty } & \underset{t\rightarrow\infty}\longrightarrow & 0,
\end{eqnarray*}
where $A_*$ is the unique minimizer in $H$ of $\J_{\mu}$. 
Moreover, $A_*$ satisfies 
\begin{multline}\label{Eq-Erreur-A-A*}
\int_{\T^d} |\na A(z) - \na A_*(z) |^2 \po \int_{\T^p\times\T^d} K(y,z) \mu(q,y)\dd q\dd y\pf \dd z \\
\leqslant\ 4 \|\na^2 A\|_\infty  \underset{y\in \T^d}\sup \int_{\T^d} |y-z|^2 K(y,z)\dd z +  2\lambda\int_{\T^d} |\na A(z)|^2 \dd z\,.
\end{multline} 
\end{thm}
Note that \eqref{Eq-Erreur-A-A*} implies that, as $\lambda$ and $\sup_{y\in \T^d} \int_{\T^d} |y-z|^2 K(y,z)\dd z$ go to zero, $A_*$ converges in $H$ to $A\nv{-\int A}$ (which corresponds to $\lambda = 0$ and $K(y,z) = \delta_y(z)$). 

\nv{The almost sure weak convergence of $\nu_t$ toward $\mu$ implies, of course, the almost sure convergence of the importance sampling estimator $\int_{\T^D} \varphi \dd \nu_t $ toward the target $\int_{\T^D} \varphi \dd \mu$ for all continuous observable $\varphi$.}
 
 \medskip
 
 The  long-time convergence of a similar ABF algorithm  has been established in \cite{LelievreAlrachid} but in a case \nv{where, instead of  its occupation measure, the process interacts with its} law  at time $t$. Rather than a self-interacting process (i.e. a single trajectory with memory), 
 this corresponds to a system of $N$ interacting particles (with no memory), and more precisely to the mean-field limit as $N$ goes to infinity of this system. \nv{The techniques to study such a non-linear process is completely different from our non-Markovian case.} Moreover,
 a result similar to Theorem \ref{TheoremTempsLong} has been established in \cite{BenaimBrehier} for a closely related self-interacting process, 
 the adaptive biasing potential algorithm.  In addition, in the recent preprint \cite{BenaimBrehierMonmarche}, a similar result is established for the ABF algorithm but when the occupation measure is not unbiased (see the discussion in Section \ref{Sec-Discuss-Real-Implement}).
 
 \medskip
 
 \nv{The previous qualitative result states that the algorithm is consistent, but gives no information on its efficiency. We now state  that the asymptotic variance of the estimators obtained from the ABF algorithm is the same as in the case of a process with constant biasing potential equal to $A_*$. More precisely, consider $X^* = (Q^*,Z^*)$ the solution of 
 \begin{eqnarray*} 
& & \left\{
\begin{array}{rcl}
\dd Q_t^* & = & - \na_q V (Q_t^*,Z_t^*) \dd t + \sqrt{2\beta^{-1}} \dd B^1_t\\
\dd Z_t^* &= &  - \na_z    V  (Q_t^*,Z_t^*)\dd t + \na_z A_*(Z_t^*) \dd t + \sqrt{2\beta^{-1}} \dd B^2_t\,,
\end{array}\right.
\end{eqnarray*}
 where $A_*$ is given by Theorem~\ref{TheoremTempsLong}, and let 
\begin{eqnarray*}
  \nu_t^* & =&  \po \int_0^t e^{-\beta A_*(Z_s)}\dd s\pf^{-1}  \int_0^t \delta_{X_s^*} e^{-\beta A_*(Z_s)} \dd s\,.
\end{eqnarray*}

\begin{thm}\label{Theorem:variance_asymptotique}
Under the settings of Theorem~\ref{TheoremTempsLong}, there exist $C>0$ such that for all $t\geqslant 0$ and all $\varphi \in \mathcal C\po \T^D\pf$,
\[\mathbb E \po |\nu_t (\varphi) - \mu(\varphi)|^2 \pf \ \leqslant \ \frac{C}{t} \|\varphi\|_\infty^2\,.\]
Moreover, $t\mathbb E \po |\nu_t (\varphi) - \mu(\varphi)|^2 \pf $ converges as $t \rightarrow +\infty$ to a limit $\sigma^2_\infty(\varphi) \in \R_+$, which is also the limit of $t\mathbb E \po |\nu_t^* (\varphi) - \mu(\varphi)|^2 \pf $.
\end{thm}
As detailed in the proof of Theorem~\ref{Theorem:variance_asymptotique} (in Section~\ref{SectionCVTempslong}), the asymptotic variance is given as follows:
\[\sigma_\infty^2 (\varphi) \ :=\ \frac{2 }{\beta  } \int_{\T^D} e^{\beta A_*\circ \xi } \left| \na   \psi \right|^2 \dd \mu \int_{\T^D}  e^{\beta A_*\circ \xi } \dd \mu \]
  where $\psi$ solves the Poisson equation
  \[\po \frac1\beta \Delta  - \na \po V- A_*\circ\xi \pf \na\pf \psi \ = \ e^{-\beta A_*\circ\xi}\po \varphi - \int_{\T^D}\varphi \dd \mu\pf \,.\]
  
 The consequences of Theorem~\ref{Theorem:variance_asymptotique} in term of efficiency of the algorithm are discussed in Section~\ref{Subsc-discussResult}.
 }

\subsection{Tensor approximation}\label{SubsectionTensor}

This section focuses on the minimization step of Algorithm~\ref{algo:ABF}. Assumption~\ref{Hypo} is enforced. Fix $\nu\in\mathcal P(\T^p \times \T^d)$. For all $f\in H$, the cost function $\J_{\nu}(f)$ defined by \eqref{EqJ1} is equal to
\begin{eqnarray}\label{Eq:LienJ-Jtilde}
\J_{\nu}(f)& =& C_{\nu} + (1+\lambda) \int_{\T^d} | F_{\nu}(z) -\na f(z)|^2 \theta_{\nu}(z)\dd z
\end{eqnarray}
with some constant $C_{\nu}$ independent from $f$ and where, for all $z\in \T^d$,
\begin{eqnarray}
\theta_{\nu}(z) &:= &\frac1{\lambda+1} \po \lambda + \int_{\T^p\times \T^d} K(y,z)\dd \nu( q,y)\pf\\
F_{\nu}(z)&: =& \frac{1}{(\lambda+1)\theta_{\nu}(z)}\int_{\T^p\times \T^d}\na_y V(q,y)  K(y,z)\dd \nu(q, y)\,. 
\end{eqnarray}
Note that, under Assumption \ref{Hypo}, $\theta_{\nu}$ is the density of a probability measure, bounded from below  by $(\lambda + \min K)/(1+\lambda)>0$. Moreover, since $K$ is smooth and bounded, so are $\theta_{\nu}$ and $F_{\nu}$. Note that neither the additive constant $C_{\nu}$ nor the multiplication by $1+\lambda$ affect 
the problem of minimizing $\J_{\nu}$. As a consequence, the unique minimizer $f_*$ of $\J_{\nu}$ on $H$ (see Proposition \ref{prop:minimJt}) is equivalently the unique minimizer of 
\begin{equation}\label{eq:defJtilde}
H \ni f \ \mapsto\ \widetilde {\mathcal J}_{\nu}(f) \ :=\  \int_{\T^d} | F_{\nu}(z) -\na f(z)|^2 \theta_{\nu}(z)\dd z.
\end{equation}
The gradient of $\widetilde \J_{\nu}$ at $f_*$ is the Helmholtz projection in $L^2(\theta_{\nu})$ of $F_{\nu}$. The Euler-Lagrange equation associated to the minimization problem of $\widetilde \J_\nu$ over $H$ is
\begin{eqnarray}\label{Eq-Euler-Lagrange}
\na\cdot\po \theta_{\nu} \po\na f_*-F_{\nu}\pf\pf & = & 0\,,
\end{eqnarray}
where $\na\cdot$ denotes the divergence operator. When $d$ is small ($d=2$ in \cite{LelievreAlrachid}), as $t$ increases, the functions $\theta_{\nu_t}$ and  $\theta_{\nu_t}F_{\nu_t}$ are updated and kept in memory on a discrete grid of dimension $M^d$ for some $M\in \mathbb{N}^*$, 
and the Euler equation is solved with standard PDE techniques. However this is not sustainable if one wants to consider a larger number of reaction coordinates. For this reason, we now present a method to approximate $f_*$ 
 by a sum of tensor products, namely by a function $f_m\in H$ which reads as follows
\[\forall z:= (z_1,\cdots, z_d)\in \T^d, \quad f_m(z) \ = \ \sum_{k=1}^m \prod_{j=1}^d r_{k,j}(z_j)\]
for some $m\in\N_*$ and some functions $r_{k,j}: \T \to \mathbb{R}$ for $1\leq j \leq d$ and $1\leq k \leq m$. See \cite{Grasedyck} for a general overview on tensor methods.

Let $g$ be a simple tensor product function, i.e. a function such that for all $z=(z_1,\cdots, z_d)\in\T^d$, $g(z) = \prod_{j=1}^d r_j(z_j)$ for some $r_1,\cdots, r_d\in H^1(\T)$. Such a simple tensor product function will be denoted herefater by 
$g = \bigotimes_{j=1}^d r_j$. 

If $g$ belongs to $H$, its (Lebesgue) integral vanishes, which is equivalent to the fact there exists $i\in\cco 1,d\ccf$ such that the (Lebesgue) integral of $r_i$ vanishes. This motivates the introduction of the following subspaces of $H$: for $i\in\cco 1,d\ccf$,   define
\begin{eqnarray*}
\Sigma_i &:= & \left\{g\in H,\ g=\bigotimes_{j=1}^d r_j\text{ with }r_j \in H^1(\T)\text{ for all $j\in\cco 1,d\ccf$ and }\int_{\T} r_i(z_i)\dd z_i =0\right \}.
\end{eqnarray*}
\begin{prop}\label{prop:minJtSigmai}
Under Assumption \ref{Hypo}, for all $\nu\in\mathcal P(\T^p \times \T^d)$, $i\in\cco 1,d\ccf$ and $f\in H$, there always exists at least one minimizer in $\Sigma_i$ to the optimization problem 
\begin{equation}\label{Eq:MinimSigma_i}
\mathop{\min}_{g\in \Sigma_i} \mathcal{J}_\nu(f + g).
\end{equation}
\end{prop}

\nv{This is proven in Section~\ref{SectionTensor}.} From Proposition \ref{prop:minJtSigmai}, the greedy algorithm described in Algorithm~\ref{algo:tenseur} below is well-defined.
\begin{algorithm}[h!]
\caption{Greedy($\nu,f_0,m$)}\label{algo:tenseur}
\begin{algorithmic}[1]

\BState \textbf{Input:}

\State Probability measure $\nu\in\mathcal P(\T^p \times \T^d)$
\State Initial guess $f_0\in H$
\State number of tensor terms  $m\in\N_*$

\medskip
\BState \textbf{Output:}
\State $f_m\in H$.

\medskip
\BState \textbf{Begin:}

\State $n=0$

\While{$n<m$}
\For{$i\in\cco 1,d\ccf$}

\State Find  $g_{n}:= \bigotimes_{j=1}^d r_{n,j}$ a minimizer of $g \mapsto \J_\nu\po f_n+g\pf$ over $g \in \Sigma_i$ (i.e. with $r_{n,j}\in H^1(\T)$ for all $1\leq j \leq d$ and $\int_\T r_{n,i} = 0$)
\State Set $f_{n+1}  =f_n  + g_{n}$ . 
\State Increment $n \leftarrow n+1$.

\EndFor \State \textbf{end for}
\EndWhile \State \textbf{end while}
 
  \medskip
\State \textbf{Return} $f_m$.
\end{algorithmic}
\end{algorithm}



In Section \ref{SectionTensor} is established the following:
\begin{thm}\label{ThmTens} Under Assumption \ref{Hypo}, let $f_*$ be the minimizer of $\J_\nu$ in $H$ and $f_m=\mathrm{Greedy} (\nu,f_0,m)$ as given by Algorithm \ref{algo:tenseur} for some $\nu\in\mathcal P(\T^p \times \T^d)$, $f_0\in H$ and $m\in\N$. Then
\[ \| f_m - f_*\|_{H^1} \ \underset{m\rightarrow+\infty}\longrightarrow \ 0\,.\]
\end{thm}

The interest of Algorithm \ref{algo:tenseur} is that at each iteration, one only has
to compute $d$ one-dimensional functions, which makes it possible to implement even if $d$
is relatively large (say $4< d < 10$). Notice that the price to pay when going from the original
problem of minimizing $\mathcal J_\nu$ over $H$ to the problem \eqref{Eq:MinimSigma_i} is that the Euler-Lagrange equations associated to the initial problem are
linear (since $\mathcal J_\nu$ is a quadratic functional) whereas the Euler-Lagrange equations
associated to \eqref{Eq:MinimSigma_i} are nonlinear. This is due to the fact that the quadratic functional is
minimized over a {\em non-linear} space in \eqref{Eq:MinimSigma_i}.

In practice, \nv{a minimizer of } $\mathcal J_{\nu}\po  f+ g\pf$ over $g = \bigotimes_{j=1}^d r_j\in \Sigma_i$  is \nv{approximated}   through the Alternating Least Square  method \cite{ALS}, 
which is a fixed point procedure on the Euler-Lagrange equation \eqref{Eq-Euler-Lagrange}: the $r_j$'s are optimized one after the other, the others being fixed, repeatedly. This amounts to solving a system of one-dimensional elliptic PDEs of the form
\begin{eqnarray}\label{eq:EDPrk}
\partial_{z_j} \po a_j \partial_{z_j} r_j\pf (z_j) - b_j(z_j) r_j(z_j) = c_j(z_j)  
\end{eqnarray}
with
\begin{eqnarray*}
a_j(z_j) & = & \int_{\T^{d-1}} \po \prod_{l\neq j} r_{l}(z_{l})\pf^2 \theta_\nu(z) \dd z_{\neq j}\\
b_j(z_j) &  =& \sum_{h\neq j} \int_{\T^{d-1}} \left|\partial_{z_h} \prod_{l\neq j} r_l(z_l)\right|^2 \theta_\nu(z) \dd z_{\neq j}\\
c_j(z_j) &= & \int_{\T^{d-1}} \po \prod_{l\neq j} r_l(z_l)\pf \partial_{z_j} \po F_{\nu,j}  \theta_\nu\pf (z)\dd z_{\neq j} - \sum_{h \neq j} \int_{\T^{d-1}}\partial_{z_h} \po \prod_{l\neq j} r_l(z_l) \pf  F_{\nu,h}(z) \theta_\nu(z) \dd z_{\neq j}\,,
\end{eqnarray*}
where $\dd z_{\neq j}$ means that all variables except the $j^{th}$ are integrated and $F_{\nu,j}$ denotes the $j^{th}$ component of $F_\nu$.
If for all $y=(y_1,\cdots, y_d),z= (z_1,\cdots,z_d)\in\T^d$, $K(y,z)=\Pi_{i=1}^d K_i(y_i,z_i)$ for some functions $K_i: \T \times \T \to \mathbb{R}$ for all $1\leq i \leq d$ (like the kernel \eqref{VonMises}), 
for $\nu = \nu_t$ given by \eqref{Eq-Empiric-Distrib},
\[(1+\lambda) a_j(z_j) \ = \ \lambda \prod_{l\neq j} \|r_l\|_{L^2(\T)}^2 + \int_{\T^D} \po \prod_{l\neq j} \int_{\T} r_l^2(z_l) K_l(y_l,z_l) \dd z_l \pf K_j(y_j,z_j)\dd \nu_t(q,y)\,,\]
which can be computed without computing $F_{\nu_t}(z)$ and $\theta_{\nu_t}(z)$ for all $z\in\T^d$ (which would be impossible in practice). The same holds for $b_j$ and $c_j$.

\subsection{The tensor ABF algorithm}\label{Sec:defTABF}

As already explained above, the main
objective of this work is to introduce a new
algorithm to adapt the standard ABF
approach to multi-dimensional reaction
coordinates. Combining Algorithms \ref{algo:ABF} and \ref{algo:tenseur}, the Tensor ABF (TABF) algorithm is  described in 
Algorithm~\ref{algo:TABF} below.  Note that, for the sake of clarity, it has been kept relatively simple. In particular,
we haven't addressed here the question of time and space discretization.

 Moreover, the proofs of
convergence of Algorithm~\ref{algo:ABF} and Algorithm~\ref{algo:tenseur}
 also have their own interest. The
convergence of Algorithm~\ref{algo:ABF} is based on the
so-called ordinary differential equation
method~\cite{BenaimLedouxRaimond}, and requires specific
contractivity bounds. The convergence of
Algorithm~\ref{algo:tenseur} is an adaptation of the proof of
convergence of greedy algorithms~\cite{CancesEhrlacherLelievre2011},
the main difficulty being to deal with the
zero average constraint in $H$.

\begin{algorithm}[h!]
\caption{TABF algorithm}\label{algo:TABF}
\begin{algorithmic}[1]

\BState \textbf{Input:}

\State Initial condition $(q_0,z_0) \in \T^p\times\T^d $
\State Brownian motion $(B_t^1,B_t^2)_{t\geqslant 0}$ on $\T^p\times\T^d$
\State Regularization parameters $K$, $\lambda$
\State Update period $T_{up}>0$, number of updates $N_{up}\in\N_*$, total simulation time $T_{tot} = T_{up} N_{up}$
\State Number of tensor terms by update  $m\in\N_*$

\medskip
\BState \textbf{Output:}
\State Estimated free energy $A_{T_{tot}} \in H$
\State Trajectory $(Q_t,Z_t)_{t\in[0,T_{tot}]}\in \mathcal C\po [0,T_{tot}],\T^p\times\T^d\pf$

\medskip
\BState \textbf{Begin:}

\State Set $(Q_0,Z_0)=(q_0,z_0)$.
\State Set $A_0(z) = 0$ for all $z\in\T^d$.
\State Set $t_k = kT_{up}$ for all $k\in\cco 0,N_{up}\ccf$.

\For{$k\in\cco 1,N_{up}\ccf$}

\State Set $A_t = A_{t_{k-1}}$ for all $t\in[t_{k-1},t_k)$.
\State Set  $(Q_t,Z_t)_{t\in[t_{k-1},t_{k}]}$ to be the solution of \eqref{EqEABF} with value $(Q_{t_{k-1}},Z_{t_{k-1}})$ at time $t_{k-1}$.
\State Set $f_m=\mathrm{Greedy}(\nu_{t_k},A_{t_{k-1}},m)$ given by Algorithm \ref{algo:tenseur} where $\nu_{t_k}$ is given by \eqref{Eq-Empiric-Distrib}.
\State Set $A_{t_k} = f_m$.
\EndFor
 
 \medskip
\State \textbf{Return} $A_{T_{tot}}$ and $(Q_t,Z_t)_{t\in[0,T_{tot}]}$.
\end{algorithmic}
\end{algorithm}

The \nv{rest of the} paper is organized as follows. \nv{Our results are discussed in Section~\ref{Subsc-discussResult}.} Section~\ref{SectionCVTempslong} is devoted to the proof of Theorems~\ref{TheoremTempsLong} \nv{and \ref{Theorem:variance_asymptotique}}. In Section~\ref{SectionTensor}, we prove Theorem~\ref{ThmTens}. Section~\ref{SubSectionDiscussion} provides a detailed 
discussion on practical considerations and possible variations of the algorithm. Finally, some  numerical experiments with the TABF algorithm are  reported in Section \ref{SectionNumerique}.


%


\subsection{\nv{Discussion on the results and efficiency}}\label{Subsc-discussResult}

\nv{First, notice that our theoretical results, Theorems~\ref{TheoremTempsLong} and \ref{ThmTens}, do not provide a full proof of convergence of the algorithm implemented in practice. Indeed, the long-time convergence is proven only in the case where the problem of minimizing $\mathcal J_{\nu}$ is exactly solved, which is not the case in Algorithm~\ref{algo:TABF}. Similarly, the convergence of the greedy tensor algorithm is proven only in the case where the problem of minimizing $\mathcal J_{\nu}(f+g)$ over single tensor terms $g = \bigotimes_{j=1}^d r_j$  is exactly solved, which is in fact not the case with the Alternating Least Square  method (see  \cite{ALS,uschmajew2012local} for convergence results for this algorihm). Besides, as already mentioned, time and space discretization errors also introduce biases in practice.  Finally,  as discussed in Section~\ref{SubSectionDiscussion}, an efficient implementation of the TABF algorithm would in fact require many other ingredients than the simple Algorithm~\ref{algo:TABF}. The present paper does not claim to fill in one leap the whole gap between theory and practice. Nevertheless, both Theorems~\ref{TheoremTempsLong} and \ref{ThmTens} are already new and non trivial results and they prove the consistency of our algorithm in some limiting
regime (perfect minimizations and negligible time and space discretization errors).


\bigskip


As Algorithm~\ref{algo:ABF} is meant to tackle metastability issues, a  natural frame to discuss its efficiency is the low temperature regime $\beta\rightarrow +\infty$. For Markov processes, obtaining an equivalent in this regime of the convergence  rate  of the law of the process toward its equilibrium  is a classical topic, but 
the case of non-Markovian self-interacting dynamics or similar stochastic algorithms is known to be much more difficult, and there are much less results. 
 Theorem~\ref{Theorem:variance_asymptotique} states that, in term of asymptotic variance, the efficiency of the adaptive scheme is approximately (as $A_*$ is close to $A$) the same as the efficiency of the importance sampling scheme based on \eqref{eq:YbiaisA}, which brings  back to the question already
discussed above of why to choose the latter as a target. 
One way to quantify the  interest of using \eqref{eq:YbiaisA} is to discuss the spectral gap of the associated infinitesimal generator: it is indeed known  that the larger the spectral  gap, the smaller the asymptotic variance, and the  quicker the convergence to equilibrium.
 In fact, at low temperature, the spectral gap of the overdamped Langevin process  is well known to scale as $\exp(-\beta c_*)$ where $c_*$ is the so-called critical depth of the potential, see \cite{Holley}. On the other hand, applying the results of \cite{LelievreTwoscale}, we see that the spectral gap of \eqref{eq:YbiaisA} can be obtained from the Poincar\'e inequality satisfied by the marginal law of the reaction coordinates on the one hand and by the conditional laws for fixed values of the reaction coordinates. The marginal law being uniform on the torus for all $\beta$, the scaling in $\beta$ of the spectral gap of \eqref{eq:YbiaisA} is given by the scaling of the Poincar\'e inequality of the conditional laws, i.e. only the ``orthogonal" metastability intervenes. The spectral gap of \eqref{eq:YbiaisA}  then scales at most as $\exp(-\beta \sup_{z\in\T^d}c_*(z))$ where $c_*(z)$ is the critical depth of $q\mapsto V(q,z)$. 
This gives a precise criterion (although difficult to use in practice) for selecting reaction coordinates: a reaction coordinate is good if $\sup_{z\in\T^d} c_*(z) < c_*$. In the toy problem studied in Section~\ref{subsec:toymodel}, for instance, $\sup_{z\in\T^d} c_*(z) = 0$ (there is no orthogonal metastability). A comparison of a classical overdamped Langevin sampler and of an ABF algorithm in this case at low temperature is given in Figure~\ref{Figure_histobet5}. We can see that, in the same physical time,  the TABF process   successfully  visits the whole space, while the classical  sampler   remains  trapped in its initial well. Of course this is not a fair comparison of the practical algorithms since the numerical cost of the adaptive algorithm is higher, but it illustrates the difference of the sampling rates of the continuous-time processes.

Besides, notice that, although Theorems~\ref{TheoremTempsLong} and \ref{Theorem:variance_asymptotique} are the first theoretical proof of the consistency of the self-interacting ABF method, this algorithm has proven to be useful and efficient for nearly 20 years  in a large number of empirical studies, see e.g. \cite{DarvePorohill,ChipotHenin,ChipotEABF} and references within.

\bigskip

Let us now discuss the efficiency of Algorithm~\ref{algo:tenseur}. A natural question that arises when it comes to tensor approximation methods is the rate of convergence of the obtained approximation to $V_{bias,t}$  as a function of  the number of tensor terms. Indeed, Theorem~\ref{ThmTens} does not provide an answer to this issue  since, a priori, the number $m$ of tensor terms required to get a correct approximation of the free energy may grow exponentially with $d$, in which case there would be no gain in using tensor formats rather than a $d$-dimensional grid. 
From a theoretical point of view, algebraic rates of convergence of greedy algorithms are proved in~\cite{temlyakov2008greedy} (see in particular \cite[Theorem 2.3.5]{temlyakov2008greedy}) under relatively mild assumptions on the regularity of the function to be approximated, 
however these rates of convergence are often observed in practice to be quite pessimistic with respect to actual rates of convergence.
For these types of algorithms, it is observed that, for elliptic problems, the rate of convergence of a tensor approximation of the solution with respect to the number of tensor terms is similar to the rate of convergence of the tensor approximation for the data of the problem. This intiuition has been rigorously proved in~\cite{dahmen2016tensor} 
in the case of a standard Laplace problem. 

\medskip

However, let us emphasise that, in fact, the convergence of $A_t$ toward $A$ (or toward $A_*$ close to $A$) is not crucial in the algorithm, 
since  the objective is to estimate the expectation of some observables. Indeed, replacing $\nabla A_t$ by any other (slowly varying) biasing force would not 
change the almost sure weak convergence of $\nu_t$ toward $\mu$ (in fact the proof of this part of Theorem~\ref{TheoremTempsLong} works as 
long as the space derivatives of $A_t$ are uniformly bounded in time, and  $\|\na A_{t_{k+1}} - \na A_{t_k}\|_{\infty}$ scales as $1/k$). As motivated 
in the introduction, the bias is chosen to target $\na A$ because the latter is a good bias (at least for well-chosen reaction coordinates). However,
as long as  the bias   helps the process  to cross some energy barriers, and thus to converge quicker to equilibrium, the fact that it is close to $\na A$ is not necessary to get 
the convergence of $\int_{\T^D} \varphi \nu_t$ toward $\int_{\T^D}\varphi \mu$. Among other consequences, it means that, 
in Algorithm~\ref{algo:TABF}, it is not necessary to chose $m$ large enough so that the convergence of the greedy tensor algorithm is achieved, any value yields a consistent algorithm.


Finally, we would like to highlight the fact that, due to the high-dimensionality of the problem, 
a standard ABF algorithm just \itshape cannot \normalfont 
be implemented in situations where the number of reaction coordinates is large, whereas the TABF method proposed here \itshape can \normalfont be used and yield significant improvements of the sampling properties of the Markov process, 
even in situations where the obtained approximation of the free energy is not very accurate. Note that, in Section~\ref{SectionNumerique}, a numerical experiment is provided where interesting non-trivial results are obtained when approximating a $5$-dimensional free energy with $m= 140$ 
tensor terms, each one-dimensional function being piecewise linear on a grid with $M=30$ points. So the total memory cost is $5m M$, orders of magnitude smaller than $M^5$.

}

\section{Proof of the long-time convergence}\label{SectionCVTempslong}

In the whole Section \ref{SectionCVTempslong} we consider the ABF process $(Q_t,Z_t,A_t)_{t\geqslant 0}$ obtaind through Algorithm~\ref{algo:ABF} (with $N_{up}=+\infty$), and Assumption \ref{Hypo} holds.

\begin{lem}\label{LemNaABorneInfini}
For all $r\in\N^*$ and all multi-index $\alpha\in \mathbb{N}^{\nv{r}}$, there exists a constant $C_\alpha>0$ such that, for all $t\geqslant 0$, $\| \partial^\alpha A_t\|_\infty \leqslant C_\alpha$.
\end{lem}

\begin{proof}
Since $\mathbb{R}_+ \ni t\mapsto A_t$ is piecewise constant, we may assume that $t=t_k = kT_{up}$ for some $k\in\N$ without loss of generality. Using the notation of Section~\ref{SubsectionTensor}, $A_t$ is then the minimizer over $H$ of 
$\widetilde{ \mathcal J}_{\nu_t}$ defined in (\ref{eq:defJtilde}). Recall that for all $f\in H$, 
\begin{eqnarray*}
 \widetilde{ \mathcal J}_{\nu_t} (f) & = & \int_{\T^d} | F_{\nu_t}(z) - \na f(z)|^2 \theta_{\nu_t}(z)\dd z.
\end{eqnarray*}
Remark that $\theta_{\nu_t}$ is bounded from below uniformly in $t$ and $z$ by $(\lambda + \min K)/(1+\lambda)>0$, 
and similarly all the derivatives in $z$ of $\theta_{\nu_t}$ and of $F_{\nu_t}$ are  bounded in $L^\infty(\T^d)$ by constants which depend on $K$ and $V$ but not on $t$. The Euler-Lagrange equation associated to the minimization of $\widetilde \J_{\nu_t}$ reads
\begin{eqnarray}\label{EqEulerJtilde}
\na\cdot \po \theta_{\nu_t} \na A_t \pf & = & \na \cdot \po \theta_{\nu_t} F_{\nu_t}\pf\,.
\end{eqnarray}
 By elliptic regularity (cf. \cite{AmbrosioCarlottoMassaccesi}), $A_t$ is thus $\mathcal C^\infty$ and, differentiating \eqref{EqEulerJtilde}, multiplying it by derivatives of $\na A_t$ and integrating, we classically get by induction that
\[ \int_{\T^d} |\partial^\alpha \na A_t |^2 \theta_{\nu_t} \ \leqslant \ C_\alpha \]
where $\alpha \in \mathbb{N}^{\nv{r}}$ is any multi-index \nv{for any $r\in\mathbb N_*$}, for some constant $C_\alpha>0$ which does not depend on $t$. Conclusion follows from Sobolev embeddings.
\end{proof}

Theorem \ref{TheoremTempsLong} will be a direct corollary of:
\begin{prop}\label{PropNutMu}
Almost surely, $\nu_t \underset{t\rightarrow \infty}{\overset{weak}{\longrightarrow}} \mu$.
\end{prop}

The proof of Proposition \ref{PropNutMu} is postponed to the end of this section. Let us prove that indeed, given the latter, Theorem \ref{TheoremTempsLong} holds:

\begin{proof}[Proof of Theorem \ref{TheoremTempsLong}]
By the  arguments of the previous proof, for all $t\geq 0$, the function $\T^d \ni z\mapsto \theta_{\nu_t}(z)$ is bounded and Lipschitz with constants which are uniform in $t$. Hence, for any $\varepsilon>0$, we can find 
$N_\varepsilon \in \mathbb{N}^*$ and
a finite set of points $z_1, \cdots, z_{N_\varepsilon} \in \T^d$ such that for all $z \in \T^d$, there exists $i_z\in\cco 1,N_\varepsilon\ccf$ such that, for all $t>0$, $|\theta_{\nu_t}(z_{i_z}) -\theta_{\nu_t}(z)|\leqslant \varepsilon$.  The same holds for $\theta_\mu$. 
On the other hand, according to Proposition \ref{PropNutMu}, almost surely,
\[\underset{i\in\cco 1,N_\varepsilon\ccf} \sup |\theta_{\nu_t}(z_i) - \theta_{\mu}(z_i)| \ \underset{t\rightarrow \infty}\longrightarrow 0,\]
so that $\|\theta_{\nu_t} - \theta_{\mu}\|_\infty$ goes to zero as $t\rightarrow \infty$. Similar arguments enable us to obtain the same results 
for all the derivatives of $\theta_{\nu_t}$ and for $F_{\nu_t}$ and all its derivatives. 
 Note that $A_*$ is the minimizer of 
\[\widetilde{ \mathcal J}_{\mu}(f) \ = \ \int_{\T^d} | F_{\mu}(z) - \na f(z)|^2 \theta_{\mu}(z)\dd z.\]
Let $t=t_n$ for some $n\in\N^*$. The associated Euler-Lagrange equations associated with the two minimization problems on $A_t$ and $A_*$ lead to
\begin{eqnarray}\label{EqEulerAtAetoile}
\na \cdot \po \theta_{\mu} \na \po A_{t} - A_* \pf\pf &=& \na\cdot \po \theta_{\nu_t} F_{\nu_t} - \theta_\mu  F_\mu - \po \theta_{\nu_t} - \theta_\mu\pf \na A_{t}\pf .
\end{eqnarray}
Multiplying this equality by $A_{t} - A_*$, integrating and using the uniform control on $\na A_{t}$ established in Lemma \ref{LemNaABorneInfini} (and the lower bound on $\theta_\mu$), we get that
\begin{eqnarray*}
\int_{\T^d} |\na \po A_{t}(z) - A_* (z)\pf|^2 \dd z  & \underset{t\rightarrow\infty}\longrightarrow & 0.
\end{eqnarray*}
More generally, differentiating \eqref{EqEulerAtAetoile}, multiplying it by derivatives of $A_{t}-A_*$, integrating and using the uniform controls of the derivatives of $A_{t}$, we obtain by induction that 
\begin{eqnarray*}
\int_{\T^d} |\na \partial^\alpha\po A_{t} -  A_*\pf(z)|^2\dd z  & \underset{t\rightarrow\infty}\longrightarrow & 0
\end{eqnarray*}
for all multi-index $\alpha \in \mathbb{N}^d$. The first statement of Theorem \ref{TheoremTempsLong} then follows from Sobolev embeddings.

Finally, inequality \eqref{Eq-Erreur-A-A*} stems from the fact that $\J_\mu(A_*)  \leqslant \J_\mu(A)$. 
More precisely, using that
\[\int_{\T^d} \na A(y) \mu(q,y)\dd q \ = \ \int_{\T^d} \na_y V(q,y) \mu(q,y)\dd q\,,\]
we get that for all $f\in H$, $\J_\mu(f) = \widehat{\mathcal J}_\mu(f) + \int|\na_y V|^2\dd \mu - \int|\na A |^2\dd\mu$  where 
\[\widehat{\mathcal J}_\mu(f) \ = \ \int_{\T^p\times\T^d\times\T^d} |\na A(y) - \na f(z) |^2 K(y, z) \dd z\dd \mu ( q,  y)  + \lambda\int_{\T^d} |\na f(z)|^2 \dd z .\]
In other words, $\J_\mu$ and $\widehat{\mathcal J}_\mu$ only differ by an additive constant, so that $A_*$ is the minimizer of $\widehat{\mathcal J}_\mu$ over $H$. Then
\begin{eqnarray*}
\lefteqn{\int_{\T^p\times\T^d\times\T^d} |\na A(z) - \na A_*(z) |^2 K(y, z) \dd z\dd \mu ( q,  y) } & & \\
& \leqslant & 2\widehat{\mathcal J}_\mu(A_*) + 2 \int_{\T^p\times\T^d\times\T^d} |\na A(y) - \na A(z) |^2 K(y, z) \dd z\dd \mu ( q,  y)\\
 & \leqslant & 2\widehat{\mathcal J}_\mu(A) +  2 \int_{\T^p\times\T^d\times\T^d} |\na A(y) - \na A(z) |^2 K(y, z) \dd z\dd \mu ( q,  y)\\
 & \leqslant & 2\lambda \int_{\T^d} |\na A(z)|^2 \dd z +  4 \int_{\T^p\times\T^d\times\T^d} |\na A(y) - \na A(z) |^2 K(y, z) \dd z\dd \mu ( q,  y)\\
  & \leqslant & 2\lambda \int_{\T^d} |\na A(z)|^2 \dd z + 4 \|\na^2 A\|_\infty^2\sup_{y\in\T^d} \int_{\T^d} |y-z|^2 K(y,z)\dd z 
\end{eqnarray*}

\end{proof}

\nv{The rest of the section is dedicated to the proof of Proposition \ref{PropNutMu} and Theorem~\ref{Theorem:variance_asymptotique}. We start with a presentation of the the so-called ordinary differential equation (ODE) method  of \cite{BenaimLedouxRaimond}, which introduces some general ideas of the proof of Proposition \ref{PropNutMu}  (although, as we will see, we are in a very simple case so that we won't really use the fully general method).}

\subsection{Time change and the ODE method}


Following an idea of \cite{BenaimBrehier}, we introduce the  (random)   time change:
\[\tau(t) \ :=\ \int_0^t e^{-\beta A_s(Z_s)}\dd s,\]
so that
\[\nu_t \ = \ \frac{1}{\tau(t)} \int_0^t \delta_{Q_s,Z_s} \tau'(s) \dd s \ = \ \frac{1}{\tau(t)} \int_0^{\tau(t)} \delta_{Q_{\tau^{-1}(s)},Z_{\tau^{-1}(s)}} \dd s\,.\]
In other words, considering the time-changed process $\overline{X}_t := \po Q_{\tau^{-1}(t)},Z_{\tau^{-1}(t)}\pf$ and its occupation measure 
\begin{equation}\label{eq:defnubar}
\bar \nu_t \ = \ \frac1t \int_0^t \delta_{\overline{X}_s}\dd s\,,
\end{equation}
then $\nu_t  = \bar \nu_{\tau(t)}$.  Since, at a fixed time $t\geqslant 0$, $A_t$ is smooth and with Lebesgue integral zero, there always exists $z\in \mathbb T^d$ such that $A_t(z) = 0$, so that 
\begin{eqnarray}\label{BorneA_t}
\| A_t \|_\infty \ \leqslant \ \sqrt{d}/2 \|\na A_t\|_\infty  
\end{eqnarray}
where we used that   $\sqrt{d}/2$ is the diameter of $\mathbb T^d$. Together with Lemma~\ref{LemNaABorneInfini}, this implies that in particular, $\tau(t)$ goes to infinity with $t$.

Denoting $S_t(x) := A_{\tau^{-1}(t)}(z)$ for all $x=(q,z) \in \T^p \times \T^d$, the inhomogeneous Markov process $\overline{X}$ solves the SDE 
\begin{eqnarray}\label{Eq:Xtbar}
\dd \overline{X}_t &=& - e^{\beta S_t\po \overline{X}_t\pf } \nabla \po V-S_t\pf \po \overline{X}_t\pf + \sqrt{2\beta^{-1} e^{\beta S_t\po \overline{X}_t\pf }  } \dd \overline{B}_t\,,
\end{eqnarray}
where $(\overline{B}_t)_{t\geqslant 0}$ is a standard Brownian motion on $\T^D$, obtained from $(B_t)_{t\geqslant 0}$ through rescaling. We denote by $(L_t)_{t\geqslant 0}$ its infinitesimal generator, defined by: for all $\varphi\in \mathcal C^2(\T^p \times \T^d) = \mathcal C^2(\T^D)$ and all $x\in \T^p \times \T^d$,
\begin{eqnarray*}
L_t \varphi(x) &=& \underset{h\rightarrow 0}\lim \frac{\mathbb E\po \varphi(\overline{X}_{t+h})\ |\ \overline{X}_t = x\pf - \varphi(x)}h
\end{eqnarray*}
whenever the limit exists. Here, 
\begin{eqnarray*}
L_t \varphi(x) & =&  \po - \na \po V - S_t\pf (x)\cdot \na \varphi(x) + \frac1\beta \Delta \varphi(x) \pf e^{\beta S_t(x)}.
\end{eqnarray*}
We denote by $(P^{(t)}_s)_{s\geqslant 0}$ the Markov semi-group generated by $L_t$ for a fixed $t$. Formally, $P_s^{(t)} = e^{sL_t}$.  
For all $t\geqslant 0$  the unique invariant measure of $L_t$ is $\mu$ (see \cite[Proposition 3.1]{BenaimBrehier})\nv{, which is a natural consequence of the fact we consider a process interacting with its \emph{unbiased} occupation measure}. From Lemma~\ref{LemNaABorneInfini} and the bound \eqref{BorneA_t}, we consider $C_0>0$ such that $S_t\in \mathcal B_{C_0}$ for all $t\geqslant 0$ where
\begin{eqnarray*}
\mathcal B_{C_0} & :=& \left\{S\in\mathcal C^\infty\po \T^D\pf,\ \int_{\T^D} S(x)\dd x = 0, \ \| S\|_{\mathcal{C}^2(\T^D)} \leqslant C_0 \right\}\,.
\end{eqnarray*}


The principle of the ODE method is the following: for large values of the time $t$, the evolution of $\bar \nu_t$ is slow (because of the $t^{-1}$ factor in (\ref{eq:defnubar})). Hence, 
for $1 \ll s \ll t$, in principle, it holds that $S_{u} \simeq S_t$ for $u\in[t,t+s]$, so that
\begin{eqnarray}\label{EqApproxEDO}
\bar \nu_{t+s} & = & \frac{t}{t+s}\bar  \nu_t + \frac{s}{s+t} \po \frac1s\int_{t}^{t+s} \delta_{{\overline{X}}_u} \dd u\pf \ \simeq \ \frac{t}{t+s}\bar  \nu_t + \frac{s}{s+t} \mu.
\end{eqnarray}
In other words, the evolution of $\bar \nu_t$ approximately follows the deterministic flow
\[\partial_t m_t \ = \ \frac1t\po \mu - m_t\pf,\]
which converges to $\mu$, so that $\bar \nu_t$ (hence $\nu_t$) should also converge to $\mu$.

\nv{In general cases of self-interacting processes, as  those studied in \cite{BenaimLedouxRaimond}, the asymptotic deterministic flow may be more complicated (see in particular \cite{BenaimBrehierMonmarche} for the ABF algorithm with the \emph{non-reweighted} occupation measure). Here, we are in a very simple case since all the generators $L_t$, $t\geqslant 0$, have the same invariant measure, so that the flow is simply a relaxation toward this equilibrium. For this reason, in order to make rigorous the previous heuristic, instead of applying the technical arguments of  \cite{BenaimLedouxRaimond}, we may use a shortcut that  yields a simpler proof and more explicit estimates (allowing in particular to tackle the question of the asymptotic variance, which may be much more intricate in  other cases), similarly to  \cite{BenaimBrehier2019} for the ABP algorithm.
}

We will need some quantitative estimates. Indeed, note that, for the approximation \eqref{EqApproxEDO} to hold, the speed of convergence of $P^{(t)}$ toward $\mu$ should be uniform in $t$, and the time evolution of $A_t$ should be controlled in some sense. As we will see below, these are direct consequences of the estimates of Lemma~\ref{LemNaABorneInfini}. 
\subsection{Preliminary estimates}\label{Sec:prelimEstimTempsLong}

For a fixed $S\in \mathcal C^\infty(\T^D)$, consider $L_S$ defined for $\varphi\in \mathcal C^\infty(\T^D)$ by
\begin{eqnarray*}
L_S \varphi(x) &=&  \po - \na \po V - S\pf (x)\cdot \na \varphi(x) + \frac1\beta \Delta \varphi(x) \pf e^{\beta S(x)}, 
\end{eqnarray*}
which is the infinitesimal generator of the SDE
\begin{eqnarray*}
\dd X^S_t & =& - e^{\beta S(X^S_t)}\na \po V  - S\pf(X^S_t) \dd t + \sqrt{2\beta^{-1} e^{\beta S(X^S_t)} } \dd B_t\,.
\end{eqnarray*}
 Denote $\po P^S_t\pf_{t\geqslant 0}$ the associated (homogeneous) semi-group and $\Gamma_S$ the associated  carr\'e-du-champs operator, defined for $\varphi,\psi\in\mathcal C^\infty(\T^D)$ and all $x\in \T^D$ by
\[\Gamma_S(\varphi,\psi)(x) \ := \ \frac12 \po L_S(\varphi\psi) - \varphi L_S \psi - \psi L_S \varphi\pf(x)  \ = \ \beta^{-1}e^{\beta S(x)} \nabla \varphi(x)\cdot \nabla \psi(x)\,,\]
and $\Gamma_S(\varphi):=\Gamma_S(\varphi,\varphi)$. By classical elliptic regularity arguments, if $\varphi\in\mathcal C^\infty(\T^D)$ then $P_t^S \varphi\in\mathcal C^\infty(\T^D)$, 
in particular $\mathcal C^\infty(\T^D)$ is a core for $L_S$, see \cite[Section 1.13]{BakryGentilLedoux}. More precisely each derivative of $P_t^S \varphi$ is uniformly bounded over all finite time interval, which 
ensures the validity of  the computations in the proofs of the next lemmas.  Integrating twice by parts, it can be easily seen that for all $\varphi,\psi\in\mathcal{C}^\infty(\T^D)$,
\[\int_{\T^D} \varphi(x) L_S \psi(x) \mu(\dd x) \ = \ \int_{\T^D} \psi(x) L_S \varphi(x) \mu(\dd x)\,,\]
in other words $L_S$ is a self-adjoint operator on $L^2(\mu)$.


\begin{lem}\label{LemLoGSobLB}
Let us assume that $D\geq 3$. Then, there exists $C_1>0$ such that for all $S\in\mathcal B_{C_0}$, $(\mu,L_S)$ satisfies a Poincar\'e inequality and a Sobolev inequality both with constant $C_1$, in the sense that for all $\varphi\in\mathcal C^\infty(\T^D)$,
\begin{eqnarray*}
 \| \varphi\|_{L^2(\mu)}^2  & \leqslant & C_1 \int_{\T^D}  \Gamma_S(\varphi) \dd \mu\\
\| \varphi\|_{L^p(\mu)}^2 & \leqslant & C_1 \po \| \varphi\|_{L^2(\mu)}^2 +   \int_{\T^D} \Gamma_S(\varphi)\dd \mu \pf, 
\end{eqnarray*}
where $p= \frac{2D}{D-2}$.
\end{lem}
\begin{proof}
For $S=0$, the first inequality is the classical Poincar\'e inequality, which holds here since the density of $\mu$ with respect to the Lebesgue measure is bounded above and below away from zero, see \cite[Proposition 5.1.6]{BakryGentilLedoux}. As a consequence, there exists $c>0$ such that for all $S\in \mathcal B_{C_0}$ and $\varphi\in\mathcal C^\infty(\T^D)$,
\[ \| \varphi\|_{L^2(\mu)}^2  \ \leqslant \ c\int_{\T^D} |\na \varphi|^2 \dd \mu \ \leqslant \ c e^{\beta C_0} \int_{\T^D} \Gamma_S (\varphi) \dd \mu\,.\]
Similarly, from the Sobolev inequality satisfied by the Lebesgue measure on $\T^D$  \cite[Section 6]{BakryGentilLedoux},
\begin{eqnarray*}
\| \varphi\|_{L^p(\mu)}^2  & \leqslant & \|\mu\|_\infty^{2/p} \| \varphi\|_{L^p(\T^D)}^2 \\
& \leqslant & C \|\mu\|_\infty^{2/p} \po \| \varphi\|_{L^2(\T^D)}^2 + \|\na \varphi\|_{L^2(\T^D)}^2 \pf \\
& \leqslant & C \|\mu\|_\infty^{2/p} \|\mu^{-1}\|_\infty^2 \po \| \varphi\|_{L^2(\mu)}^2 + \|\na \varphi\|_{L^2(\mu)}^2 \pf\\
& \leqslant & C e^{\beta C_0} \|\mu\|_\infty^{2/p} \|\mu^{-1}\|_\infty^2 \po \| \varphi\|_{L^2(\mu)}^2 + \int_{\T^D} \Gamma_S(\varphi)\dd \mu\pf\,.
\end{eqnarray*}
\end{proof}

These inequalities, in turn, yield the following estimates:

\begin{lem}\label{LemEstimInterm}
There exist $C_2>0$ such that, for all $S\in\mathcal B_{C_0}$,  $t\geqslant0$ and $\varphi\in\mathcal C^\infty(\T^D)$, 
\begin{eqnarray*}
\| P_t^S \Pi \varphi \|_{L^2(\mu)} & \leqslant & e^{- t/C_2} \|  \Pi \varphi \|_{L^2(\mu)}\\
\| P_t^S \varphi \|_{\infty} & \leqslant & \frac{C_2}{\min(1,t^{d/2})} \|   \varphi \|_{L^2(\mu)}\\
\| \na P_t^S \varphi \|_{\infty} & \leqslant & \frac{C_2}{\min(1,\sqrt t)} \|   \varphi \|_{\infty},
\end{eqnarray*}
with $\Pi \varphi := \varphi-\int_{\T^D} \varphi \dd \mu$.
\end{lem}
\begin{proof}
The first estimate is a usual consequence of the Poincar\'e inequality, see \cite[Proposition 5.1.3]{BakryGentilLedoux}. The second one, namely the ultracontractivity of the semi-group, is a consequence of the Sobolev inequality (see \cite[Theorem 6.3.1]{BakryGentilLedoux}). The last one can be established thanks to the Bakry-Emery calculus (see \cite[Section 1.16]{BakryGentilLedoux} for an introduction), by showing that $L_S$ satisfies a curvature estimate, as we now detail. We would like to compare $|\na P_t \varphi|^2$ and $P_t (\varphi^2)$. A seminal idea of the Bakry-Emery calculus is that quantities of the form $\Theta(P_t \varphi)$ and $P_t\Theta(\varphi)$, where $\Theta$ is some operator can be linked through the interpolation  $P_{t-s}\Theta(P_s \varphi)$, $s\in[0,t]$, so that $\Theta(P_t \varphi) - P_t\Theta(\varphi) = \int_0^t \partial_s \po P_{t-s} \Theta (P_s \varphi)\pf \dd s$. When differentiating with respect to $s$, we obtain quantities of the form $-2P_{t-s}\Gamma_{\Theta}(P_s \varphi)$ for some operator $\Gamma_{\Theta}$, which is of a form similar to the interpolation ($\Theta$ being replaced by $\Gamma_\Theta$).

 More precisely, when $\Theta(\varphi) = \varphi^2$, then $\Gamma_{\Theta}$ is the usual carr\'e-du-champ operator, and when $\Theta(\varphi)=|\na \varphi|^2$  we end up with
\begin{eqnarray*}
\Gamma_{\na,S}(\varphi) &= & \frac12 L_S\po |\na \varphi|^2\pf - \na \varphi\cdot \na L_S \varphi\,,
\end{eqnarray*}
for $\varphi\in\mathcal C^\infty(\T^D)$. Writing $[\varphi,\psi]=\varphi\psi-\psi\varphi$, we compute
 \begin{eqnarray*}
\Gamma_{\nabla,S} (\varphi) & = & \sum_{i=1}^D \po \Gamma_S(\partial_{x_i}\varphi) + \partial_{x_i} \varphi [\partial_{x_i},L_S]\varphi \pf\\
& \geqslant & \sum_{i=1}^D  \Big[ \beta^{-1} e^{-\beta \|S\|_\infty}|\na \partial_{x_i} \varphi|^2 - \beta^{-1} e^{\beta \|S\|_\infty}|\na\partial_{x_i} (V-S)||\na \varphi||\partial_{x_i} \varphi|\\
& & - \beta |\partial_{x_i} S|e^{\beta \|S\|_\infty} |\partial_{x_i}\varphi||\nabla \po V - S\pf \cdot \na \varphi + \frac1\beta \Delta \varphi|\Big]\\
& \geqslant & -c |\na \varphi|^2
\end{eqnarray*}
for some $c>0$ which is uniform over $S\in\mathcal B_{C_0}$. Now, following \cite[Lemma 4]{MonmarcheGamma}, we want to consider the interpolation between $\alpha(t) |\na P_t \varphi|^2 + (P_t \varphi)^2$ and $\alpha(0)P_t|\na \varphi|^2+P_t (\varphi^2)$ for some $\alpha$ with $\alpha(0) = 0 <\alpha(t)$. For  fixed $\varphi\in\mathcal C^\infty(\T^D)$, $x\in\T^D$ and  $t\geqslant 0$, we set for all $s\in[0,t]$
\[\Psi(s) \ = \ \alpha(s) P_{t-s}^S  |\na P_{s}^S \varphi|^2(x)  + e^{\beta C_0} P_{t-s}^S\po P_{s}^S \varphi\pf^2(x)\]
with $\alpha(s) = (1-\exp(-2ct))/c$, so that
\begin{eqnarray*}
\partial_s \Psi(s) & = & P_{t-s}^S \po -2\alpha(s) \Gamma_{\na,S} + \alpha'(s) |\na \cdot|^2 -2 e^{\beta C_0} \Gamma_S \pf\po P^S_{s} \varphi\pf(x)\\
& \leqslant &  \po 2\alpha(s) c + \alpha'(s)  -2  \pf P_{t-s}^S |\na P^S_{s} \varphi|^2(x) \ = \ 0\,.
\end{eqnarray*}
In particular,
\[\alpha(t) |\na P_t^S \varphi|^2(x) \ \leqslant \ \Psi(t) \ \leqslant \ \Psi(0) \ = \ e^{\beta C_0} P_t^S \varphi^2(x) \ \leqslant \ e^{\beta C_0}\|\varphi\|_\infty^2\]
which yields the desired estimate.
\end{proof}

\begin{lem}\label{Lem:estimPoisson}
There exists $C_3>0$ such that for all $S\in\mathcal B_{C_0}$, the operator $R_S$ defined for all $\varphi\in\mathcal{C}^\infty(\T^D)$ by
\begin{eqnarray*}
R_S \varphi & = & -\int_0^\infty P_t^S \Pi \varphi \dd t
\end{eqnarray*}
satisfies $L_SR_S=R_SL_S= \Pi$ and, for all $\varphi\in\mathcal{C}^\infty(\T^D)$,
 \begin{eqnarray}\label{EqEstimateQB}
 \| R_S \varphi \|_\infty + \|\na R_S \varphi\|_\infty\  + \|\Delta R_S \varphi\|_\infty & \leqslant & C_3 \|\varphi\|_\infty.
 \end{eqnarray}\end{lem}
\begin{proof}
We follow the proof of \cite[Section 5.2 and Lemma 5.1]{BenaimLedouxRaimond}. First, from Lemma~\ref{LemEstimInterm} (and using the fact that $\|P_t^S \varphi\|_\infty \leqslant \|\varphi\|_\infty$ for all $t\geqslant 0$),
\begin{eqnarray*}
  \int_0^\infty \| P_t^S\Pi \varphi\|_\infty dt & \leqslant &   \int_0^1 \|\Pi \varphi\|_\infty dt +   \int_1^\infty \| P_t^S\Pi \varphi\|_\infty dt \\
  & \leqslant & 2\|\varphi\|_\infty +   C_2 \int_1^\infty \| P_{t-1}^S\Pi \varphi\|_{L^2(\mu)} dt \\
    & \leqslant & 2\|\varphi\|_\infty +    C_2 \int_1^\infty  e^{-(t-1)/C_2} \|\Pi \varphi\|_{L^2(\mu)} dt \\
    & \leqslant & 2(1+C_2^2)\|\varphi\|_\infty\,,
\end{eqnarray*} 
and similarly, using  the fact that $\|\nabla P_t^S \Pi \varphi\|_\infty = \| \nabla P_1 P_{t-1}^S
 \Pi \varphi \|_\infty \le C_2 \| P_{t-1}^S \Pi \phi \|_\infty$ for $t\geqslant 1$,
\begin{eqnarray*}
  \int_0^\infty \| \nabla P_t^S \Pi \varphi\|_\infty dt & \leqslant &   \int_0^1 \frac{C_2}{\sqrt t}   \|  \Pi \varphi\|_\infty dt + C_2 \int_1^\infty    \|   P_{t-1}^S\Pi \varphi\|_\infty dt   \\  
    & \leqslant & C_2 \po \nv{4} +2(1+C_2^2)\pf \|\varphi\|_\infty\,.
\end{eqnarray*} 
In particular $R_S \varphi$ and $\nabla R_S\varphi$ are well defined in $L^\infty(\T^D)$ for $\varphi\in\mathcal C^\infty(\T^D)$. Moreover, using the fact that, from Lemma~\ref{LemEstimInterm}, $\| P_t^S \Pi \varphi \|_\infty  \leqslant C_2 e^{-(t-1)/C_2}\|  \Pi \varphi \|_{L^2(\mu)}  \rightarrow 0$ as $t\rightarrow +\infty$,
\begin{eqnarray*}
L_S R_S \varphi &  = & -\int_0^\infty L_S P_t^S \Pi \varphi dt \\
& = & -\int_0^\infty \partial_t\left( P_t^S \Pi \varphi\right) dt  \ = \ \Pi \varphi\,.
\end{eqnarray*}
The case of $R_SL_S$ is similar: since $\mu$ is invariant for $L_S$, $L_S\Pi=L_S = \Pi L_S$, and thus $P_t^S \Pi L_S = P_t^S L_S\Pi = \partial_t(P_t^S\Pi\varphi)$ for all $t\geqslant 0$.

As a consequence,
 \[|\Delta R_S \varphi|\ \leqslant \ e^{\beta C_0} |e^{\beta S}\Delta R_S \varphi| \ \leqslant \  e^{\beta C_0} \po \| e^{\beta S} \nabla (V-S)\cdot \na R_S \varphi\|_\infty +  \|\Pi \varphi\|_\infty \pf \ \leqslant \ C_3\|\varphi\|_\infty\]
 for some $C_3>0$ uniform over $S\in\mathcal B_{C_0}$, which yields the desired result.
\end{proof}
%
 \begin{lem}\label{LemQBQbprime}
 There exist $C_4>0$ such that for all $S_1,S_2\in\mathcal B_{C_0}$ and $\varphi\in\mathcal C^\infty(\T^D)$,
 \begin{eqnarray*}
 \| R_{S_1}\varphi - R_{S_2} \varphi \|_\infty + \nv{\|\na  R_{S_1}\varphi - \na  R_{S_2} \varphi \|_\infty } & \leqslant & C_4 \|\na S_1 - \na S_2\|_\infty \|\varphi\|_\infty.
 \end{eqnarray*}
 \end{lem}
 \begin{proof}
 From $R_S L_S = \Pi$, 
 \begin{eqnarray*}
 \po R_{S_1} - R_{S_2}\pf L_{S_1} + R_{S_2}\po L_{S_1} - L_{S_2}\pf & = & 0.
 \end{eqnarray*}
 Multiplying this equality by $R_{S_1}$ on the right, and using that $R_S \Pi = R_S$, we get  for all $\varphi \in \mathcal C^\infty(\T^D)$,
  \begin{eqnarray*}
 \po R_{S_1} - R_{S_2}\pf \varphi  & = & R_{S_2}\po L_{S_2} - L_{S_1}\pf R_{S_1} \varphi\,.
 \end{eqnarray*}
 \nv{Thus, from \eqref{EqEstimateQB},
 \begin{eqnarray*}
\lefteqn{\| \po R_{S_1} - R_{S_2}\pf \varphi\|_\infty + \|\na  R_{S_1}\varphi - \na  R_{S_2} \varphi \|_\infty}\\
 & \leqslant & C_3 \|\po L_{S_2} - L_{S_1}\pf R_{S_1} \varphi\|_\infty\\
&\leqslant & C_3 \| e^{\beta S_2} \na (S_1- S_2)\cdot \na R_{S_1} \varphi\|_\infty + C_3 \|( 1 - e^{\beta(S_2-S_1)}) L_{S_1} R_{S_1} \varphi\|_\infty \\
& \leqslant & C_3 e^{\beta C_0}\|\na(S_1-S_2)\|_\infty \|\na R_{S_1}\varphi\|_\infty + C_3 e^{2\beta C_0} \| S_2-S_1\|_\infty \|L_{S_1} R_{S_1} \varphi\|_\infty \,.
 \end{eqnarray*}
Conclusion follows from \eqref{EqEstimateQB}. Indeed, notice that $S_1-S_2$ is a continuous function with integral zero, so that there exist $x\in\mathbb T^D$ such that $(S_1-S_2)(x)=0$, and then for all $y\in \mathbb T^D$
\[|(S_1-S_2)(y)| \leqslant |x-y|\|\na(S_1-S_2)\|_\infty\]
so that $\| S_1-S_2 \| _\infty \leqslant \sqrt D \|\nabla S_1 - \nabla S_2\|_\infty$.
 }

  \end{proof}
  
  \begin{lem}\label{LemQAkQAr}
  There exists $C_5>0$ such that for all $k\geqslant 1$  and $\varphi\in\mathcal C^\infty(\T^D)$
   \begin{eqnarray*}
 \| R_{A_{t_k}}\varphi - R_{A_{t_{k-1}}} \varphi\|_\infty   & \leqslant & \frac{C_5}{k}  \|\varphi\|_\infty.
 \end{eqnarray*}
  \end{lem}
  \begin{proof}
 From Lemma \ref{LemNaABorneInfini}, $A_t \in\mathcal B_{C_0}$ for all $t\geqslant 0$, so that Lemma \ref{LemQBQbprime} applies. It remains to obtain a bound on $\| \na {A_{t_k}} - \na {A_{t_{k-1}}}   \|_\infty$. 
 In this proof, to simplify the notation, we write $\theta_k = \theta_{\nu_{t_{k}}}$ and $F_k = F_{\nu_{t_{k}}}$.  Denoting by
  \[ m := \frac{\int_{t_k}^{t_{k+1}} \delta_{(Q_s,Z_s)} e^{\beta A_s(Z_s)} \dd s } {\int_{t_k}^{t_{k+1}} e^{\beta A_s(Z_s)} \dd s } \quad \mbox{ and }\quad p := \frac{\int_{t_k}^{t_{k+1}} e^{\beta A_s(Z_s)} \dd s } {\int_{0}^{t_{k+1}} e^{\beta A_s(Z_s)} \dd s },   \]
  it holds that
  \begin{eqnarray*}
  \nu_{t_{k+1}} = (1-p) \nu_{t_k} + p m.
  \end{eqnarray*}
  In particular, for some $c,c'>0$, for all $z\in \T^d$ and $k\in\N$,
  \[| \theta_{k+1}(z) - \theta_{k}(z)|\ = \frac{p}{1+\lambda} \left| \int_{(q,y)\in \T^p \times \T^d} K(z,y)(\dd  m (q,y)- \dd \nu_{t_k}(q,y))\right| \ \leqslant \ cp \ \leqslant \  \frac{c'}{k},\]
  where we used that $A_s\in\mathcal B_{C_0}$ for all $s\in[0,t_{k+1}]$. The same argument also works for the derivatives of $\theta_{\nu_t}$, for $F_{\nu_t}$ and its derivatives, so that for any multi-index $\alpha \in \mathbb{N}^d$, there exists a constant $C_\alpha$ such that for all $k\geqslant 1$,
  \[ \| \partial^\alpha F_{k+1} -\partial^\alpha F_{k}\|_\infty + \| \partial^\alpha \theta_{k+1} -\partial^\alpha \theta_{k}\|_\infty \ \leqslant \ \frac{C_\alpha}{k}. \] 
  Now, from the Euler equations satisfied by $A_{t_k}$ and $A_{t_{k+1}}$, we get
  \begin{equation}\label{EqEulerAkAkp1}
  \na\cdot \po \theta_k \na \po A_{t_k}-A_{t_{k+1}}\pf \pf \ = \  \na\cdot \po \na A_{t_{k+1}} \po \theta_{k+1} - \theta_k\pf \pf - \na \cdot \po \theta_{k+1} F_{k+1} - \theta_k F_k\pf.
  \end{equation}
  Multiplying this equation by $A_{t_k}-A_{t_{k+1}}$, integrating and using Lemma \ref{LemNaABorneInfini} and the lower bound on $\theta_{k}$, we get
  \begin{eqnarray*}
  \int_{\T^d} |\na \po A_{t_k}-A_{t_{k+1}}\pf(z)|^2 \dd z & \leqslant & \frac{c}{k^2}
  \end{eqnarray*}
  for some $c>0$. Next, differentiating \eqref{EqEulerAkAkp1}, multiplying it by derivatives of $A_{t_k}-A_{t_{k+1}}$, integrating and using by induction the previous estimates,  we obtain in fact that 
    \begin{eqnarray*}
  \int_{\T^d} |\na \partial^\alpha\po  A_{t_k}-A_{t_{k+1}}\pf(z)|^2 \dd z & \leqslant & \frac{c_\alpha}{k^2}
  \end{eqnarray*}
  for some $c_\alpha>0$ for all $\alpha\in \mathbb{N}^d$, and Sobolev embeddings then yield the conclusion.
  \end{proof}

\subsection{Proof of the main results}

\nv{In this section we denote $\nu(\varphi) = \int_{\T^D} \varphi \dd \nu$ the expectation of an observable $\varphi$ with respect to a probability measure $\nu$.

\begin{proof}[Proof of Proposition \ref{PropNutMu}]
In the following, we use the same notation $C$ for various constants. For all $t\geqslant 0$ and $ \varphi\in \mathcal C^\infty\po\T^D\pf$,  
\begin{eqnarray*}
 \nu_t (\varphi)  - \mu(\varphi)  &= & \frac{1}{\tau(t)} \int_0^t e^{-\beta A_s(Z_s)} \Pi \varphi (X_s) \dd s  \\
 & = & \frac{1}{\tau(t)} \int_0^t e^{-\beta A_s(Z_s)} L_{A_s} R_{A_s} \varphi (X_s) \dd s\,.
\end{eqnarray*}
To alleviate notations, write $\tilde \varphi_s  = R_{A_s} \varphi$. For $k \in \N$ and $t\in [t_k,t_{k+1})$, from \eqref{EqEABF} by It\^o's formula,
\begin{eqnarray*}
\tilde \varphi_t(X_t) - \tilde \varphi_{t_k}(X_{t_k}) & = & \int_{t_k}^t   e^{-\beta A_s(Z_s)} L_{A_s} \tilde \varphi_s (X_s) \dd s +  \sqrt{\frac{2}{\beta }}\int_{t_k}^t \na \tilde \varphi_s (X_s) \dd B_s\,,
\end{eqnarray*}
so that
\begin{eqnarray}\label{eq:*}
\qquad \tau(t) \po  \nu_t (\varphi)  - \mu (\varphi)  \pf  & = & \int_0^t e^{-\beta A_s(Z_s)} L_{A_s} \tilde\varphi_s (X_s) \dd s \ = \ \tilde \varphi_t(X_t) - \tilde \varphi_{0}(X_{0}) \\
& &  + \sum_{0<t_k\leqslant t} \po \tilde \varphi_{t_{k-1}}(X_{t_{k}}) - \tilde \varphi_{t_{k}}(X_{t_{k}})\pf 
 - \sqrt{\frac{2}{\beta }}\int_{0}^t \na \tilde \varphi_s (X_s) \dd B_s\,.\nonumber 
\end{eqnarray}
Recall that, from Lemma~\ref{LemNaABorneInfini}, there exists $C>0$ such that almost surely $\tau(t) \geqslant t/C$ for all $t\geqslant 0$. Together with
 Lemma~\ref{LemQAkQAr},  we get that there exists $C>0$ such that, almost surely, for all $t> 0$ and $ \varphi\in \mathcal C^\infty\po\T^D\pf$,
\begin{equation}\label{eq:**}
 \frac{1}{\tau(t)} \po |\tilde \varphi_t(X_t) - \tilde \varphi_{0}(X_{0}) | + \sum_{0<t_k<t}  | \tilde \varphi_{t_{k}}(X_{t_{k}}) - \tilde \varphi_{t_{k-1}}(X_{t_{k}})|  \pf    \leqslant   \frac{C \ln(1+t)}{t} \|\varphi\|_\infty\,.
\end{equation}
  Moreover, applying It\^o's isometry, 
\begin{eqnarray*}
\mathbb E\po \left| \int_{0}^t \na \tilde \varphi_s (X_s) \dd B_s\right|^2 \pf & = & \mathbb E\po  \int_{0}^t \left| \na \tilde \varphi_s (X_s)\right|^2 \dd  s \pf \\
& \leqslant & t C_3^2 \|\varphi\|_\infty^2 \,,
\end{eqnarray*}
where we used \eqref{EqEstimateQB}. As a consequence, there exists $C>0$ such that 
\begin{eqnarray}\label{eq:***}
\mathbb E \po \left|  \nu_t ( \varphi)  - \mu(\varphi)  \right|^2 \pf  & \leqslant& \frac{C }{t} \|\varphi\|_\infty^2
\end{eqnarray}
for all $t> 0$ and $ \varphi\in \mathcal C^\infty\po\T^D\pf$. As in the proof of \cite[Lemma 5.1]{BenaimBrehier2019}, this implies the almost sure weak convergence of $\nu_t$ to $\mu$ as follows. Indeed, for all $r>0$, the Borel-Cantelli Lemma yields the almost sure  convergence of $\nu_{\exp(nr)} (\varphi)$ toward $\mu(\varphi)$ as $n\rightarrow +\infty$, $n\in\N$, so that
\begin{equation}\label{eq:BC}
\mathbb P \po \forall k\in \mathbb N_*,\ \nu_{\exp(n/k)} (\varphi) \underset{n\rightarrow +\infty}\longrightarrow \mu(\varphi)\pf \ = \ 1\,.
\end{equation}
 Moreover, using the almost sure bounds $t/C \leqslant \tau(t) \leqslant Ct$  and $ \|\exp(-\beta A_s)\|_\infty \leqslant C $ for some $C>0$, we get that  there exists $C,C'>0$ such that for all $t\geqslant s>0$,
\begin{eqnarray*}
|\nu_t (\varphi) - \nu_s (\varphi) | & \leqslant & C\po \left|\frac{1}{\tau(t)} - \frac{1}{\tau(s)}\right| s  + \frac{  |t-s|}{\tau(t)}\pf \|\varphi\|_\infty \\ 
& \leqslant & C\po  \frac{|\tau(t) - \tau(s)|s}{\tau(t)\tau(s)}  + \frac{  |t-s|}{\tau(t)}\pf \|\varphi\|_\infty \\
& \leqslant &  C'\frac{|t-s|}{t}\|\varphi\|_\infty\,.
\end{eqnarray*} 
As a consequence, for all $ \varphi\in \mathcal C^\infty\po\T^D\pf$, $t\mapsto \nu_{\exp(t)}(\varphi)$ is almost surely $C'\|\varphi\|_\infty$-Lipschitz, and in particular
\begin{equation}\label{eq:BC2}
\mathbb P \po \forall k\in \mathbb N_*,\forall t\geqslant 0,\ |\nu_{\exp(t)}(\varphi)-\nu_{\exp(\lfloor tk\rfloor/k)}(\varphi)|\leqslant \frac{C'\|\varphi\|_\infty}{k}\pf \ = \ 1\,.
\end{equation}
The almost sure convergence of $\nu_t (\varphi) $ to $\mu (\varphi)$ for a given $\varphi$  then follows from  
\begin{multline*}
\left\{ \nu_t (\varphi) \underset{t\rightarrow +\infty}\longrightarrow \mu(\varphi)\right\} \ = \   \bigcap_{k\in\N_*} \left\{\limsup_{t\rightarrow\infty}| \nu_{\exp(t)} - \mu(\varphi)| \leqslant \frac{C'\|\varphi\|_\infty}{k}\right\}  \\
\  \supset \ \bigcap_{k\in\N_*}\po  \left\{\sup_{t\geqslant 0}|\nu_{\exp(t)}(\varphi)-\nu_{\exp(\lfloor tk\rfloor/k)}(\varphi)|\leqslant \frac{C'\|\varphi\|_\infty}{k}\right\} \cap \left\{ \nu_{\exp(\lfloor tk\rfloor/k)} \underset{t\rightarrow +\infty}\longrightarrow \mu(\varphi) \right\} \pf \,,
\end{multline*}
the last event having probability $1$ from \eqref{eq:BC} and \eqref{eq:BC2}. 
 Considering a sequence $(\varphi_k)_{k\in \N}$ of $\mathcal C^\infty\po\T^D\pf$ functions that is dense in $\mathcal C\po\T^D\pf$, we get that
\[\mathbb P \po \nu_t (\varphi_k) \underset{t\rightarrow +\infty}\longrightarrow \mu( \varphi_k) \ \forall k \in \N\pf \ = \ 1\,,\]
so that almost surely $\nu_t$ converges weakly to $\mu$ as $t\rightarrow +\infty$. This concludes the proof of Proposition \ref{PropNutMu}, hence of Theorem \ref{TheoremTempsLong}.
\end{proof}

\begin{proof}[Proof of Theorem~\ref{Theorem:variance_asymptotique}]

The first claim of the  theorem  has already been established in the proof of Proposition~\ref{PropNutMu}, see \eqref{eq:***}. Fix $\varphi\in \mathcal C^\infty \po \T^D\pf$. We have seen in the proof of Proposition~\ref{PropNutMu} (see \eqref{eq:*} and \eqref{eq:**}) that, denoting  $\tilde \varphi_s = R_{A_s}\varphi$,
\begin{eqnarray*}
   \nu_t (\varphi)  - \mu(\varphi)     & = & \varepsilon_t   -\frac1{\tau(t)} \sqrt{\frac{2}{\beta }}\int_{0}^t \na \tilde \varphi_s (X_s) \dd B_s
\end{eqnarray*}
for some $\varepsilon_t$ such that almost surely $|\varepsilon_t|\leqslant C\ln(1+t)/t$ for all $t>0$ for some $C>0$. The martingale part having zero expectation, the bias is bounded as
 \[\left| \mathbb E (\nu_t (\varphi)  - \mu(\varphi)  ) \right|^2 = \left|\mathbb E(\varepsilon_t)\right|^2 \leqslant \frac{C^2\ln^2(1+t)}{t^2} = \underset{t\rightarrow +\infty} o\po \frac1t\pf\,.\]
 In other words, the asymptotic mean-square error is only due to the asymptotic variance, which is itself only due to the martingale part of  $ \nu_t (\varphi)  - \mu(\varphi) $.

 Remark that, as a corollary of Theorem~\ref{TheoremTempsLong}, $\|A_t - A_*\|_\infty \rightarrow 0$ almost surely. Together with the uniform bounds of Lemma~\ref{LemNaABorneInfini}  and the weak convergence of $\nu_t$ to $\mu$, we get that
 \begin{eqnarray*}
 \frac{t}{\tau(t)} \  = \ \frac{t-\sqrt t}{\tau(t)} + \underset{t\rightarrow +\infty} o\po 1\pf
  & = &  \frac{1}{\tau(t)}\int_{\sqrt t}^t e^{-\beta A_*(Z_s)}  e^{ \beta A_*(Z_s)} \dd s  + \underset{t\rightarrow +\infty} o\po 1\pf\\
    & = &  \frac{1}{\tau(t)}\int_{\sqrt t}^t e^{-\beta A_s(Z_s)}  e^{ \beta A_*(Z_s)} \dd s  + \underset{t\rightarrow +\infty} o\po 1\pf\\
 & = &  \frac{1}{\tau(t)}\int_0^t e^{-\beta A_s(Z_s)}  e^{ \beta A_*(Z_s)} \dd s  + \underset{t\rightarrow +\infty} o\po 1\pf\\
 & \underset{t\rightarrow +\infty}\longrightarrow & \mu \po e^{\beta A_* \circ \xi} \pf  \ =:\ \kappa
 \end{eqnarray*}
 almost surely. As a consequence,
 \begin{eqnarray*}
 t \mathbb E\po |\nu_t (\varphi)  - \mu(\varphi)  |^2\pf & = & \frac{2t}{\beta}   \mathbb E\po  \frac{1}{\tau^2(t)}\left|\int_{0}^t \na \tilde \varphi_s (X_s) \dd B_s\right|^2\pf  + \underset{t\rightarrow +\infty} o\po 1\pf\\
 & = & \frac{2\kappa^2}{\beta t}   \mathbb E\po   \left|\int_{0}^t \na \tilde \varphi_s (X_s) \dd B_s\right|^2\pf  + \underset{t\rightarrow +\infty} o\po 1\pf \\
 & = & \frac{2\kappa^2}{\beta t}   \int_0^t \mathbb E\po    \left| \na \tilde \varphi_s (X_s)\right|^2\pf  \dd s + \underset{t\rightarrow +\infty} o\po 1\pf \,.
  \end{eqnarray*}
  From Lemma~\ref{LemQBQbprime}, $\|\na \tilde \varphi_s - \na R_{A_*} \varphi\|_\infty \leqslant C_4 \|\na A_s - \na A_*\|_\infty \|\varphi\|_\infty$, which together with Theorem~\ref{TheoremTempsLong} and \eqref{EqEstimateQB} yields
   \begin{eqnarray*}
 t \mathbb E\po |\nu_t (\varphi)  - \mu(\varphi)  |^2\pf 
 & = & \frac{2\kappa^2}{\beta t}   \int_0^t \mathbb E\po    \left| \na   R_{A_*} \varphi (X_s)\right|^2\pf  \dd s + \underset{t\rightarrow +\infty} o\po 1\pf \\
 & = & \frac{2\kappa^2}{\beta t}  \mathbb E\po  \int_0^t    \left| \na   R_{A_*} \varphi (X_s)\right|^2 e^{-\beta (A_s(Z_s)-A_*(Z_s))}  \dd s \pf + \underset{t\rightarrow +\infty} o\po 1\pf \\
  & = & \frac{2\kappa }{\beta  }  \mathbb E\po \nu_t \po     \left| \na   R_{A_*} \varphi  \right|^2 e^{\beta A_* \circ \xi} \pf \pf + \underset{t\rightarrow +\infty} o\po 1\pf \\
 & \underset{t\rightarrow +\infty} \longrightarrow & \frac{2\kappa }{\beta  } \mu \po     \left| \na   R_{A_*} \varphi  \right|^2 e^{\beta A_* \circ \xi} \pf\,.
  \end{eqnarray*}
  In other words, the asymptotic variance is 
  \[\frac{2 }{\beta  } \mu  \po e^{\beta A_*\circ \xi } \left| \na   \psi \right|^2\pf  \mu  \po e^{\beta A_*\circ \xi } \pf \]
  where $\psi = R_{A_*} \varphi $ solves $L_{A_*}\psi = \Pi\varphi$, which reads
  \[\po \frac1\beta \Delta  - \na \po V- A_*\circ\xi \pf \na\pf \psi \ = \ e^{-\beta A_*\circ\xi}\Pi\varphi\,.\]
  
  It remains to see that we get the same formula for the asymptotic variance of $\sqrt t \po \nu_t^*(\varphi) - \mu(\varphi)\pf$. The computations are similar, so we only sketch the main points. Denoting $\tau_*(t) = \int_0^t \exp(-\beta A_*(Z_s^*))\dd s$, as in Proposition~\ref{PropNutMu},
  \begin{eqnarray*}
\tau_*(t)\po  \nu_t^*( \varphi)  - \mu(\varphi)\pf   &= &  \int_0^t e^{-\beta A_*(Z_s^*)} \Pi \varphi (X_s^*) \dd s  \\
 & = &   \int_0^t e^{-\beta A_*(Z_s^*)} L_{A_*} R_{A_*} \varphi (X_s^*) \dd s\\
 &= & R_{A_*} \varphi(X_t^*) - R_{A_*} \varphi (X_{0}^*)  - \sqrt{\frac{2}{\beta }}\int_{0}^t \na R_{A_*} \varphi(X_s) \dd B_s\,.
\end{eqnarray*}
Again, from the almost sure bound $\tau_*(t)\geqslant t/C$ for some $C>0$, It\^o's isometry and the bounds on $R_{A_*} \varphi$ given by \eqref{EqEstimateQB}, we get that there exists $C>0$ such that for all $t>0$ and $\varphi \in \mathcal C(\T^D)$,
\begin{eqnarray*}
\mathbb E \po \left|  \nu_t^*\po \varphi\pf  - \mu\po \varphi\pf  \right|^2 \pf  & \leqslant& \frac{C }{t} \|\varphi\|_\infty^2\,.
\end{eqnarray*}
The proof that this implies the almost sure weak convergence of $\nu_t^*$ to $\mu$ is similar to the end of the proof of Proposition~\ref{PropNutMu}. As a corollary,
\[\frac{t}{\tau_*(t)} \  = \ \nu_t^* \po e^{\beta A_*\circ \xi} \pf \ \underset{t\rightarrow +\infty}\longrightarrow \ \kappa\,. \]
Finally,
 \begin{eqnarray*}
 t \mathbb E\po |\nu_t (\varphi)  - \mu(\varphi)  |^2\pf & = & \frac{2t}{\beta}   \mathbb E\po  \frac{1}{\tau^2_*(t)}\left|\int_{0}^t \na R_{A_*} \varphi (X_s^*) \dd B_s\right|^2\pf  + \underset{t\rightarrow +\infty} o\po 1\pf\\
 & = & \frac{2\kappa^2}{\beta t}   \mathbb E\po   \left|\int_{0}^t \na R_{A_*} \varphi(X_s^*) \dd B_s\right|^2\pf  + \underset{t\rightarrow +\infty} o\po 1\pf \\
 & = & \frac{2\kappa^2}{\beta t}   \int_0^t \mathbb E\po    \left| \na R_{A_*} \varphi(X_s^*)\right|^2\pf  \dd s + \underset{t\rightarrow +\infty} o\po 1\pf \\
 & = & \frac{2\kappa^2}{\beta t}     \mathbb E\po  \tau_*(t) \nu_t^*\po    \left| \na R_{A_*} \varphi \right|^2 e^{\beta A_* \circ \xi}\pf \pf  \dd s + \underset{t\rightarrow +\infty} o\po 1\pf \\
  & \underset{t\rightarrow +\infty} \longrightarrow & \frac{2\kappa }{\beta  } \mu \po     \left| \na   R_{A_*} \varphi  \right|^2 e^{\beta A_* \circ \xi} \pf\,.
  \end{eqnarray*}

\end{proof}
}

\section{Consistency of the tensor approximation}\label{SectionTensor}

This section is devoted to the proof of Propositions \ref{prop:minimJt} and \ref{prop:minJtSigmai} and Theorem~\ref{ThmTens}. In all this section, Assumption \ref{Hypo} holds and we write $\theta = \theta_\nu$, $F=F_\nu$, $\mathcal J = \mathcal J_\nu$ for some fixed $\nu \in \mathcal P(\T^p \times \T^d)$.
Recall that $F,\theta\in\mathcal C^\infty(\T^d)$, that $\theta$ is a positive probability density on $\T^d$, and that the minimizers of $\J$ in $H$ are exactly the minimizers of $\widetilde{\mathcal J}$ in $H$, where for all $f\in H$, 
\[\widetilde{\mathcal J}(f) = \int_{\T^d} | F(z) -\na f(z)|^2 \theta(z)\dd z \,,\]
the link between $\mathcal J$ and $\widetilde{\mathcal J}$  begin given by \eqref{Eq:LienJ-Jtilde}. 

Since $\theta$ is bounded from above and below by positive constants,   the weighted spaces $L^2(\T^d; \theta)$ and $H^1(\T^d; \theta)$ are equal to the flat spaces $L^2(\T^d; \dd z)$ and $H^1(\T^d; \dd z)$. We endow $H$ (whose definition is given in (\ref{eq:defH}),
with the  norm
\[\| f\| = \sqrt{(1+\lambda)\int_{\T^d} |\na f(z)|^2 \theta(z)\dd z},\]
which is indeed a norm, equivalent to the usual $H^1$ norm from the Poincar\'e-Wirtinger inequality: there exists $C>0$ such that for all $f\in H$,  
\begin{eqnarray*}
\int_{\T^d} f^2(z) \theta(z) \dd z &\leqslant& \|\theta\|_\infty \int_{\T^d} f^2(z)  \dd z\\
& \leqslant& C\|\theta\|_\infty \int_{\T^d} |\na f(z)|^2  \dd z \\
& \leqslant& C\|\theta\|_\infty\|\theta^{-1}\|_\infty \int_{\T^d} |\na f(z)|^2 \theta(z) \dd z.
\end{eqnarray*}
 The  scalar product associated with $\| \cdot\|$ is denoted by $\langle\cdot\rangle$. The choice of such a norm is motivated by the fact that, denoting by $\mathcal J'$ the differential of $\J$, then for all $f,g\in H$,
\begin{eqnarray*}
\J(f) & = & \J(0) + \J'(0)\cdot f + \| f\|^2\\
\J'(f) \cdot g & = &  \J'(0)\cdot g +  2\langle f,g\rangle \,.
\end{eqnarray*}
Proposition \ref{prop:minimJt} is then a direct consequence of the strict convexity of $\J$. The unique minimizer $f_*$ of $\J$ over $H$ \nv{being a minimizer of $\widetilde{\mathcal J}$, it satisfies $\widetilde{\mathcal J}'(f_*)=0$, which reads}  
\[\forall g\in H\,,\qquad \int F(z)\cdot \na g(z) \theta(z) \dd z \ = \ \int \na f_*(z) \cdot \na g(z) \theta(z) \dd z \]
\nv{Moreover, using that $\mathcal J'(f_*)=0$, we get that $\J'(0)\cdot g = -  2\langle f_*,g\rangle $ for all $g\in H$ and then $ \J(0) = \J(f_*) -\J'(0)\cdot f_* -\| f_*\|^2 = \J(f_*)+\|f_*\|^2$. As a consequence, for all $g\in H$,
\begin{eqnarray}\label{EqJg}
   \J(g) & = & \J(0) + \J'(0)\cdot g + \| g\|^2\ = \  \J(f_*) + \| f_* - g\|^2\label{EqJg}\,.
\end{eqnarray}
}

\begin{proof}[Proof of Proposition \ref{prop:minJtSigmai}]
Let $f\in H$ and $i\in\cco 1,d\ccf$. If $\mathcal J\po  f \pf= {\inf} \{\mathcal J\po  f  + g\pf,\ g\in \Sigma_{i}\}$ then the result is correct since $0\in\Sigma_i$ is a minimizer over $\Sigma_i$. Suppose now that $0$ is not a minimizer, 
i.e. that $\mathcal J\po  f \pf > {\inf} \{\mathcal J\po  f  + g\pf,\ g\in \Sigma_{i}\}$ and consider a minimizing sequence $(g^{(l)})_{l\in \mathbb{N}}$ in $\Sigma_i$ 
such that $\mathcal J\po  f  + g^{(l)}\pf$ converges to ${\inf} \{\mathcal J\po  f  + g\pf,\ g\in \Sigma_{i}\}$ as $l$ goes to infinity. For $l$ large enough, 
$\mathcal J\po  f  + g^{(l)}\pf  < \mathcal J\po  f  \pf $ so that $g^{(l)}\neq 0$, and thus up to an extraction we suppose that $g^{(l)}\neq 0$ for all $l\in\N$. 
Moreover the sequence is bounded in $H^1$ and thus, up to the extraction of a subsequence, we suppose that it weakly converges in $H^1$ to some $g^*\in H$. The function $H \ni g\mapsto \mathcal J\po  f  + g\pf$ being convex on $H$, 
it is weakly lower semi-continuous, so that 
\begin{eqnarray*}
\mathcal J\po g_* \pf & \leqslant & \underset{g\in \Sigma_{i}}{\inf} \mathcal J\po  f  + g\pf.
\end{eqnarray*}
For all $l\in \mathbb{N}$, there exist $r_1^{(l)}, \cdots, r_d^{(l)} \in H^1(\T)$ such that $g^{(l)} = \bigotimes_{j=1}^d r_{j}^{(l)}$. Since $g^{(l)}\neq 0$,
we can normalize the $r_j^{(l)}$'s so that $\| r_j^{(l)}\|_{L^2(\T)} = 1$ for all $j\neq i$ and $l\in \mathbb N$. As a consequence, up to the extraction of a subsequence, 
for all $1\leq j \neq i \leq d$, there exists $r_j^{*} \in L^2(\T)$ such that the sequence $(r_j^{(l)})_{l\in \mathbb{N}}$ weakly converges to $r_j^*$ in $L^2(\T)$. 
Now, since the sequence $g^{(l)}$ is bounded in $H$ and
  \begin{eqnarray*}
\| \na g^{(l)}\|_{L^2(\T^d)}^2 & = & \sum_{j=1}^d \| \partial_{z_j} r_j^{(l)} \|_{L^2(\T)}^2 \prod_{h\neq j} \| r_h^{(l)}\|_{L^2(\T)}^2,
\end{eqnarray*} 
we get that the sequence $(r_i^{(l)})_{l\in\mathbb{N}}$ is bounded in $H^1(\T) $ (since $\int_\T r_i^{(l)} = 0$ for all $l\in \mathbb{N}$). Thus, up to the extraction of another subsequence, 
the sequence $(r_i^{(l)})_{l\in\mathbb{N}}$ weakly converges in $H^1(\T)$ to some $r_i^{*}\in H^1(\T)$ such that $\int_\T r_i^* = 0$. 
From \cite[Lemma 2]{LBLM}, $(g^{(l)})_{l\in \mathbb{N}}$ converges in the distributional sense to $ \bigotimes_{j=1}^d r_j^*$. Thus, $g^* = \bigotimes_{j=1}^d r_j^*$ and since $g^* \neq 0$, this implies that for all 
$1\leq j \leq d$, $r_j^*\neq 0$. Finally, 
since 
$$
\| \na g^*\|_{L^2(\T^d)}^2  =  \sum_{j=1}^d \| \partial_{z_j} r_j^* \|_{L^2(\T)}^2 \prod_{h\neq j} \| r_h^*\|_{L^2(\T)}^2
$$
is a finite quantity, this implies that for all $1\leq j \neq i \leq d$, 
$$
\| \partial_{z_j} r_j^* \|_{L^2(\T)} \leq \frac{\| \na g^*\|_{L^2(\T^d)}^2}{\prod_{h\neq j} \| r_h^*\|_{L^2(\T)}^2} < +\infty,
$$
and thus $r_j^* \in H^1(\T)$. This implies that $g_*\in \Sigma_i$ and yields the desired result. 
\end{proof}

\begin{remark}
 The problem would be ill-posed if we were to try and minimize $\J(f  + g - \int_{\T^d} g)$ over all
 $g \in \Sigma:= \left\{ r_1\otimes \cdots \otimes r_d, \; r_j\in H^1(\T) \mbox{ for all } 1\leq j \leq d \right\}$. This is the reason why we introduced the condition that one of the $r_j$'s has zero mean. 
 Indeed, consider the situation where $d=2$, $f=0$ and $F(z) = \po a'(z_1),b'(z_2)\pf$ for some smooth functions $a,b: \T \to \mathbb{R}$ with zero mean. 
 Then, the minimum of $\J$ over $H$ is $0$, and only attained at $f^*(z_1,z_2) = a(z_1)+b(z_2)$, which is not of the form $g-\int_{\T^d} g$ for some $g\in \Sigma$. Nevertheless, the sequence $(g^{(l)})_{l\in\mathbb{N}^*}$ defined by:
 for all $l\in \mathbb{N}^*$, $g^{(l)} = r_1^{(l)} r_2^{(l)}$ with
\[ r_1^{(l)}(z_1) = 1 + \frac{a(z_1)}{l},\hspace{25pt} r_2^{(l)}(z_2) = l + b(z_2)\] 
is a minimizing sequence. Indeed, the sequence $\left(g^{(l)} - \int_{\T^d} g^{(l)}\right)_{l\in \mathbb{N}^*}$ weakly converges to $f^*$. 
In other words, the set $\left\{g - \int_{\T^d} g,\ g\in \Sigma \right\}$ is not weakly closed in $H$.
\end{remark}

\bigskip

Proposition \ref{prop:minJtSigmai} proves that all the iterations of Algorithm~\ref{algo:tenseur} are well-defined. 
In the following, we consider a sequence $(f_n)_{n\in\N}$ given by the latter and $g_n = f_{n+1}-f_n$ for $n\in\N$. The general idea of the 
proof of Theorem \ref{ThmTens} is that, if the sequence $(f_n)_{n\in\mathbb{N}}$ converges to some $f_\infty$ in $H$, it holds that 
$\J'(f_\infty)\cdot g = 0$ for all $g\in\cup_{i=1}^d\Sigma_i$, and a density argument enables to conclude. Nevertheless, remark that 
$\Sigma_i$ is not a vector space and that its elements all have null integral, so that one should be careful. In the following,  we essentially adapt the arguments of \cite{CancesEhrlacherLelievre2011}.

\begin{lem}\label{LemControlNorme}
For all $g\in \cup_{i=1}^d \Sigma_i$ and $n\in\N$,
\begin{eqnarray*}
\left| \J'(f_{n}) \cdot g \right| & \leqslant & 6  \| g\| \sum_{j=0}^{d-1} \| g_{n+j}\|.
\end{eqnarray*}
\end{lem}
\begin{proof}
Let $i\in\cco 1,d\ccf$ be such that $g\in\Sigma_i$. Let $n_i\in \cco n,n+d-1\ccf$ be such that $g_{n_i}\in\Sigma_i$. We bound, first,
\begin{eqnarray*}
\left| \J'(f_{n}) \cdot g \right| & \leqslant & \left| \J'(f_{n_i}) \cdot g  \right| +2\left|  \langle g,f_{n_i}-f_{n}\rangle\right|.
\end{eqnarray*}
The second term of the right hand side is bounded by  $2\| g\| \sum_{j=n}^{n_i-1} \| g_{j}\|$. To deal with the first one, note that, even though $g_{n_i}$ is a minimizer of $\mathcal J(f_{n_i} + \cdot)$ over $\Sigma_{i}$,
it is not necessarily true that $\J'(f_{n_i}+g_{n_i})\cdot g = 0$, since $\Sigma_{i}$ is not a vector space. 
 We follow the proof of \cite[Proposition 3.3]{CancesEhrlacherLelievre2011}. 
 By convexity of
\[t\in \R\ \mapsto\ \psi(t) \ :=\ \J\po  f_{n_i} + g + t(g_{n_i}-g)\pf,\]
and since $t=0$ minimizes $\psi(t)$, we get
\[ \psi'(0) \ \leqslant\ \psi(1)-\psi(0) \ \leqslant\ 0,\]
which reads
\begin{eqnarray*}
\J'\po  f_{n_i} + g \pf \cdot g  & \geqslant & \J'\po  f_{n_i} + g \pf \cdot g_{n_i}.
\end{eqnarray*}
Hence, 
\begin{eqnarray*}
-\J'(f_{n_i}) \cdot g     & = &  - \J'(f_{n_i}+g) \cdot g     + 2 \| g\|^2 \\
& \leqslant & - \J'(f_{n_i}+g) \cdot g_{n_i}     + 2 \| g\|^2 \\
& \leqslant & - \J'(f_{n_i}+g_{n_i}) \cdot g_{n_i}  + 2 \langle g_{n_i} ,g_{n_i}-g \rangle   +2 \| g\|^2.  
\end{eqnarray*}
Now, $1$ being a minimizer over $\R$ of $ t\mapsto \J(f_{n_i}+tg_{n_i})$,  $\J'(f_{n_i}+g_{n_i}) \cdot g_{n_i}  = 0$, so that
\begin{eqnarray*}
-\J'(f_{n_i}) \cdot g     & \leqslant &  2   \po \| g_{n_i} \|^2 +   \| g\| \| g_{n_i}\| + \| g\|^2 \pf.
\end{eqnarray*}
When applied to $\widetilde g = \pm \|g_{n_i}\|   g / \|g\|$, this inequality yields
\begin{eqnarray*}
\left| \J'(f_{n_i}) \cdot g \right|  & \leqslant & 6  \| g \| \| g_{n_i}\|,
\end{eqnarray*}
which concludes the proof.
\end{proof}

\begin{prop}
Let $f_*$ be the unique minimizer of $\J$ over $H$. Then
\begin{eqnarray*}
\| f_n - f_*\| & \underset{n\rightarrow \infty}\longrightarrow & 0.
\end{eqnarray*}
\end{prop}
\begin{proof}
As in the previous proof, 
\[0 \ = \ \J'(f_n+g_n)\cdot g_n \ = \ 2(1+\lambda)\int \na g_n \cdot  (\na f_{n+1} - F)\theta = 0\]
for all $n\in\N$, so that, 
\begin{equation}\label{eq:decroissance}
\J(f_n) - \J(f_{n+1}) \ = \ \J(f_{n+1}- g_n )- \J(f_{n+1}) \ = \ \| g_n\|^2.
\end{equation}
In particular, since $\po \mathcal J(f_n)\pf_{n\geqslant 0}$ is a decreasing sequence bounded from below,
\begin{equation}\label{eq:finitesum}
 \sum_{n\geqslant 0} \| g_n\|^2 < \infty .
\end{equation}
Together with Lemma \ref{LemControlNorme} and the fact $\J'(f_*) = 0$, this implies that for all $g\in \mbox{\rm Span}\po \cup_{i=1}^d \Sigma_i\pf$,
\begin{eqnarray*}
2\langle f_*-f_n, g\rangle \ = \ \J'(f_*)\cdot g - \J'(f_n)\cdot g& \underset{n\rightarrow \infty}\longrightarrow & 0.
\end{eqnarray*}
Now, for all $r_1,\cdots, r_d \in \mathcal C^{\infty}(\T)$, denoting by $h:= \bigotimes r_i \in \mathcal{C}^{\infty}(\T^d)$ (note that we do not have necessarily that $\int_{\T^d} h = 0$), we can write
\begin{eqnarray*}
h - \int_{\T^d} h  
& = & \sum_{j=1}^d \po r_j - \int_{\T} r_j  \pf \po \prod_{l<j} \int_{\T} r_l  \pf \prod_{l>j} r_l\,,
\end{eqnarray*}
which proves that $h-\int_{\T^d} h \in \mbox{\rm Span}\po \cup_{i=1}^d \Sigma_i\pf$. As a consequence, 
\[ (1+\lambda)\int_{\T^d}  \theta  \, \na (f_*-f_n)\cdot \nabla h\ = \ \langle f_*-f_n, h-\int_{\T^d} h\rangle \underset{n\rightarrow \infty}\longrightarrow  0\,, \]
 As a consequence, the limit $f_\infty$ of any convergent (in the weak sense in $H$) subsequence of $(f_n)_{n\in\N}$  necessarily satisfies that 
$$
\int_{\T^d}  \theta  \, \na (f_*-f_\infty)\cdot \nabla h = 0,
$$
for any tensor product function $h = \bigotimes r_i$, with $r_i\in \mathcal C^\infty (\T)$ for all $1\leq i \leq d$. By \cite[Lemma 2.1]{CancesEhrlacherLelievre2011}, this implies that $f_\infty=f_*$. 
On the other hand, since $\left(\| f_n\|\right)_{n\in \mathbb{N}}$ is bounded, all its subsequences admits weak convergent subsequences, and the fact they all have the same limit $f_*$ proves that the whole sequence $(f_n)_{n\in \mathbb{N}}$ 
weakly converges in $H$ to $f_*$. In particular
\begin{eqnarray}\label{Eqfnf*}
\langle f_*-f_n, f_*\rangle  & \underset{n\rightarrow \infty}\longrightarrow  & 0.
\end{eqnarray}
Thus, it only remains to prove that $\left(\langle f_* - f_n, f_n\rangle\right)_{n\in \mathbb{N}}$ also converges to zero as $n$ goes to infinity to obtain the \itshape strong \normalfont convergence of the sequence $(f_n)_{n\in \mathbb{N}}$ to $f_*$ in $H$. 
From Lemma~\ref{LemControlNorme}, 
\begin{eqnarray*}
2| \langle f_* - f_n, f_n\rangle| & = & \left| \J'(f_n)\cdot f_n \right|\\
& \leqslant & \sum_{k=0}^n \left|\J'(f_n)\cdot g_k \right|\\
& \leqslant & 6 \po \sum_{k=0}^{n-1} \| g_k\|\pf \sum_{j=n}^{n+d-1} \| g_j\|\\
& \leqslant &  6 \sqrt{ n d a_n \sum_{k=0}^\infty \| g_k\|^2} .
\end{eqnarray*}
with $a_n = \sum_{j=n}^{n+d-1} \| g_j\|^2$. Using (\ref{eq:finitesum}), since $\sum_{n\in \mathbb{N}} a_n \leq d \sum_{n\in\mathbb{N}} \|g_n\|^2 <\infty$, there exists an extracted subsequence $(n_k)_{k\geqslant1}$ such that $(n_k a_{n_k})_{k\geq 1}$
converges to $0$ as $k$ goes to infinity. As a consequence, $(\langle f_* - f_{n_k}, f_{n_k}\rangle)_{k\geq 1}$ goes to zero as $k\rightarrow \infty$, and thus so does $(\| f_{n_k}-f_*\|)_{k\geq 1}$ by \eqref{Eqfnf*}. 
Finally, from~\eqref{EqJg}, the sequence $(\| f_{n}-f_*\|)_{n\in\mathbb{N}}$ is non-increasing, so that the whole sequence goes to zero. Hence the result.
\end{proof}

\section{Discussion and variations}\label{SubSectionDiscussion} 

For the sake of clarity, the TABF algorithm defined in Section~\ref{Sec:defTABF} has been kept relatively simple, and there is obviously room for many variations or fine-tuning.  We list here a few of them.

\subsection{Extended ABF}\label{Sec-EABF}
Consider general reaction coordinates $\xi:\T^D \rightarrow \mathcal M$ where $\mathcal M$ is a submanifold of $\mathbb{R}^d$ or $\T^d$. In the Extended ABF (EABF) algorithm introduced in \cite{LelievreEABF} (see also \cite{Lesage-EABF,ChipotEABF}), 
the state space is extended to $\T^D \times \mathcal M$ with the addition of auxiliary variables (or fictitious particles) $z\in \mathcal M$, and the potential $V$ on $\T^D$ is extended to a potential $\tilde V$ on $\T^D\times\mathcal M$ as
\[\tilde V(q,z) \ = \ V(x) + \frac{1}{2\sigma^2} \po \text{dist}_{\mathcal M}\po \xi(q),z\pf\pf^2, \quad \forall (q,z)\in \T^D \times \mathcal M, \]
for some small parameter $\sigma>0$, where $\text{dist}_\mathcal M$ stands for the distance on $\mathcal M$. 
The reaction coordinates on the extended space are then defined by $\tilde \xi(q,z) = z$, which means the framework considered in the present paper is general for
the EABF algorithm. If $(Q,Z)$ is distributed according to $\mu_{\tilde V,\beta}$, the law of $Z$ is obtained from the law of $\xi(Q)$ through a Gaussian convolution of
variance $\sigma^2/\beta$.  There are several practical advantages to EABF:
\begin{itemize}
\item In the potential $\tilde V$, in the case where $\text{dist}_\mathcal M$ is an Euclidean distance, the $z_i$'s for $i\in\cco 1,d\ccf$ appear 
in separate terms of a sum, they are not directly coupled. As a consequence, in the EABF case, $\mu_{A,\beta}$ should be, in some sense, closer to the product of 
its marginal (namely, at equilibrium, the $Z_i$'s should be closer to be independent) than in the non-extended ABF case. In \cite{ChipotEABF}, this was a crucial point since the 
density $\mu_{A,\beta}$ was approximated by a tensor product. But even in our case where the approximation as a sum of tensor product is made at the level of $A$, we can expect this form of $\tilde V$ to improve the approximation.
\item After convolution, the so-called mean-force $\na_z A$
is smoother than the initial mean force in the non-extended ABF case. Since it varies less, its estimation is expected to be easier.
\end{itemize}
That being said, the tensorized ABF introduced above can also be straightforwardly  extended  to a general ABF framework, without extended coordinates. 

\subsection{Non-periodic reaction coordinates} \label{Sec-Reflection}

In general, $\mathcal M$ may be different from $\T^d$. If it has boundaries, for instance if $\mathcal M=[0,1]^d$, the definition of the algorithm is the same except that the diffusion \eqref{EqEABF} is reflected at the boundaries of $\mathcal M$. The proof of well-posedness and convergence of the tensor algorithm, i.e. Theorem~\ref{ThmTens}, is unchanged. The proof of the long-time convergence of the idealized algorithm, i.e. Theorem~\ref{TheoremTempsLong}, is similar up to technical considerations in particular to take into account boundary conditions in Section \ref{Sec:prelimEstimTempsLong}.

Moreover, $\mathcal M$ may  not be compact, for instance $\mathcal M=\R^d$\nv{, with $U$ satisfying suitable growth conditions at infinity}. Since the Lebesgue measure has not a finite mass, a confining biasing potential has to be added to the adaptive biasing potential, see \cite[Section 1.2]{LelievreABF}, in which case the law of $\xi(X_t)$ does not converge to a uniform law (flat histogram) but to a target unimodal law on $\R^d$.

\subsection{Real implementation}\label{Sec-Discuss-Real-Implement}

The algorithm really implemented for the numerical experiments in Section \ref{SectionNumerique} differs from the theoretical Algorithm \ref{algo:TABF} in the following points:

\begin{enumerate}
\item Time and space are discretized. The SDE \eqref{EqEABF} is replaced by an Euler-Maruyama scheme 
with some timestep $\delta t$ and the time integral in \eqref{Eq-Empiric-Distrib} is replaced by a discrete sum with a 
timestep $\Delta t$ (not necessarily small; it can be of the order of the decorrelation length of the process $(Q_t,Z_t)_{t\geqslant 0}$). 
The one-dimensional functions in the tensor terms are restricted to be continuous piecewise linear, determined by their value on a discrete grid 
with $q\in\N_*$ points, so that solving \eqref{eq:EDPrk} amount to solve a $q\times q$ linear system. In particular, the discrete space 
interpolation plays a role similar to the regularization kernel $K$ which is no more necessary, hence is discarded.

\item In fact, it is not necessary to re-weight the occupation distribution, namely 
\eqref{Eq-Empiric-Distrib} can be replaced by $\nu_t = 1/t\int_0^t \delta_{(Q_s,Z_s)}\dd s$. In that case,
$\nu_t$ is expected to converge to $\mu_{V-\tilde A_*,\beta}$ for some $\tilde A_*$ instead of $\mu_{V,\beta}$ but the conditional law of $Q$ 
given $Z=z$ is the same for these two laws. Since the free energy only depends on these conditional laws, $A_t$ is still 
expected to converge to $\tilde A_*$, that should be close to the true free energy in a sense similar 
to \eqref{Eq-Erreur-A-A*}. This is clear in the mean-field limit of the algorithm, where no regularization is needed so
that $\tilde A_*=A$ (see \cite{LelievreABF}). It is more difficult to establish for the self-interacting ABF process, but in parallel of the present paper it 
has been done in \cite{BenaimBrehierMonmarche}.  We tried numerically both cases, and the results were similar. The results presented in Section \ref{SectionNumerique} 
are obtained with the full (non reweighted) occupation distribution. 

\item Instead of a single particle, in practice, several replicas of the process \eqref{EqEABF} are simulated
in parallel. Denoting $N$ the number of replicas and $(Q^i_t,Z^i_t)_{t\geqslant 0}$ the $i^{th}$ replica, $i\in \cco 1,N\ccf$, the total empirical distribution of the system is 
\begin{eqnarray}\label{Eq-nutildeN}
\tilde \nu_{N,t} \ = \ \frac1{N\lfloor t/\Delta t\rfloor} \sum_{i=1}^N \sum_{k=1}^{\lfloor t/\Delta t\rfloor} \delta_{(Q_{k\Delta t}^i,Z_{k\Delta t}^i)}\,.
\end{eqnarray}
The replicas all use the same bias $A_t$ obtained from this empirical distribution by minimizing $\J_{\tilde \nu_{N,t}}$ at times $t=t_k = k T_{up}$.

\end{enumerate}

\subsection{Other possible simple variations} \label{Sec-Variations}

\begin{enumerate}
\item  From the biased trajectory $(Q_t,Z_t)_{t\geqslant 0}$ provided by the TABF algorithm, in order 
to compute expectations with respect to the target Gibbs measure $\mu = \mu_{V,\beta}$, an alternative to the 
reweighting step \eqref{eq:ImportanceSampling} is the following. Remark that only the $Z$ variable is biased, so that for all $z\in \T^d$, 
the conditional expectations $\int_{\T^p} f(q,z) \mu(q,z) \dd q /\int_{\T^p}  \mu(q,z) \dd q$ 
can be estimated without re-weighting. On the other hand, the marginal law of $Z$ is estimated by $\exp(-\beta A_{T_{tot}})/\int_{\T^d} \exp(-\beta A_{T_{tot}}(z))\dd z$.

\item The bias update period $T_{up}$ and the number $m$ of tensor products added at each update in Algorithm 
\ref{algo:TABF}, instead of having fixed values, could be adaptively chosen. For instance, the bias could be 
updated when the histogram of the reaction coordinates have reached some stability, and $m$ could be the lowest integer $n\in\N$ such that $\J_\nu(A_{t_k}+f_{n-d}) -  \J_\nu(A_{t_k}+f_{n}) \leqslant \varepsilon$ for some threshold $\varepsilon>0$. 

\item A time-dependent weight in the definition \eqref{Eq-Empiric-Distrib} of $\nu_t$ 
(or, in practice, in \eqref{Eq-nutildeN} for $\tilde \nu_{N,t}$) can be added in such a way that old samples have less influence than new ones since they are more biased toward the initial distribution.

\item The regularization kernel $K$ and parameter $\lambda$ may depend on time. Indeed, as time goes, the size 
of the sample increases. Since the problem of minimizing $\J_{\nu_t}$ is in practice solved on a finite dimension space, for a time large 
enough the regularization is actually not necessary anymore and the minimization problem with $K(z,y) = \delta_z(y)$ and $\lambda = 0$ is well-posed.

\item It is possible to use the tensor approximation only as a correction of the classical ABF, or 
more precisely of the Generalized ABF (GABF) algorithm proposed in \cite{ChipotEGABF} where the bias is just a 
sum of one-dimensional functions. Namely, for all time $t>0$, for $j\in\cco 1,d\ccf$, let 
\begin{eqnarray*}
\alpha_{j,t}(z_j)& = & \int_{\T^p\times\T^d} \partial_{y_k} V(q,y) K(y_j,z_j) \dd \nu_{t}(q,y) \\
\beta_{j,t}(z_j) &=&  \int_{\T^p\times\T^d} K(y_j,z_j) \dd \nu_{t}(q,y) \,.
\end{eqnarray*}
These functions can be recorded on $d$ one-dimensional grids and are easily updated on the fly. 
Denoting $ \gamma_{t,j}(z_j)= \1_{\beta_{j,t}(z_j) > s}\alpha_{j,t}(z_j)/\beta_{j,t}(z_j)$ for some burn-in time 
$s>0$, let  $A_{t,j}(z_j) = \int_{0}^{z_j} \gamma_{t,j}(z)\dd z$ if the $j^{th}$ reaction coordinate $z_j$ lies in $\R$ and  
$A_{t,j}(z_j) = \int_{0}^{z_j} \gamma_{t,j}(u)\dd u - z_j\int_0^1 \gamma_{t,j}(u)\dd u $ if $z_j$ lies in $\T$ (so that, in both cases, $\partial_{z_j} A_{t,j}$
is the Helmoltz projection in $L^2(\dd z_j)$ of $\gamma_{t,j}$). Then, at time $t$, in the dynamics \eqref{EqEABF}, use the bias $\na_z A_t$ with $A_{t}(z) = \sum_{j=1}^d A_{t_k,j}(z_j) + f_m(z)$ 
where $f_m$ is a tensor approximation obtained through Algorithm \ref{algo:tenseur} of the minimizer of $H \ni f\mapsto \J_{\nu_{t_k}}(f- \sum_{j=1}^d A_{t_k,j})$, where $t_k = \sup\{t_{k'}<t,\ k'\in\N\}$ is the last update time.

\item For $k\in\N_+$, denote by $\J_{k}^\lambda$ the function given by \eqref{EqJ1} for some $\lambda>0$ with $\nu=\nu_{t_{k}}$. Rather than setting $A_{t_k}$ to be the minimizer of $\J_{k}^\lambda$, we can set it 
to be $A_{t_{k-1}} + f$ where $f$ is the minimizer of 
\begin{eqnarray}\label{eq:Jalternative}
H\ni f & \mapsto & \J_{k}^0\po  A_{t_{k-1}} + f \pf + \lambda\int_{\T^d} |\na_z f(z)|^2 \dd z\,,
\end{eqnarray}
 the difference being that $A_{t_{k-1}}$ does not  appear in the last regularization term any more. Note that, when $\lambda = 0$, there is no difference. 
 When $\lambda>0$, the theoretical results of Section \ref{SectionTensor}, i.e. the well-posedness of the tensor approximation, can be straightforwardly adapted. 
 The long-time behaviour study of Section \ref{SectionCVTempslong} should be similar, although a bit more troublesome since  $A_{t_k}$ would not depend only on the empirical 
 distribution $\eta_{t_k}$ but also on the previous bias $A_{t_{k-1}}$. On the other hand, remark that $0$ is a minimizer of \eqref{eq:Jalternative} if and only if $A_{t_k}$ is a minimizer of $\J_k^0$. As a consequence, 
 the long-time limit of $A_t$ should be the minimizer of $\J_\mu$ with $\lambda=0$ which, in view of \eqref{Eq-Erreur-A-A*}, advocates for this alternative form of cost function.

\item Since the bias is stored in memory in a tensor form, it is possible to use at some times a compression algorithm (see \cite{HOSVD}) to reduce the number of tensor terms, if needed.
\end{enumerate}

\subsection{Some limitations and perspectives}

A practical limitation observed in the algorithm is the following. Recall that $d$ is too large to keep in memory  the empirical measure  
on a grid by simply recording how many times each $d$-dimensional cell has been visited by the process, as in the classical ABF algorithm. Instead, 
the sequence $(Z_{k\Delta t},\nabla_z V(Q_{k\Delta t},Z_{k\Delta t}))_{k\in\N}$ is kept in memory for some $\Delta t>0$, and thus computing 
an expectation with respect to $\nu_t$ has a numerical cost proportional to $t$. Such integrals are computed when solving the one-dimension 
equations \eqref{eq:EDPrk}, which have to be solved repeatedly at each addition of a tensor term to the bias. As $t$ grows, the update of the bias gets 
numerically more expensive. We list here some possible directions to address this question. The analysis of these variations is beyond the reach of the present work.

\begin{enumerate}
 \item  At the beginning of Algorithm \ref{algo:tenseur}, a clustering or quantization algorithm (see \cite{PagesQuantization}) can be used to reduce the memory $(Z_{k\Delta t},\nabla_z V(Q_{k\Delta t},Z_{k\Delta t}))_{k\in\cco 1,t_n/\Delta t\ccf}$ to fewer points.

\item Another way to deal with this problem would be to use a fixed small size for the memory. For instance, at an update time 
$t_k$, the empirical measure used to define $\mathcal J_{\bar \nu_{t_k}}$ could be 
\begin{eqnarray*}
\bar \nu_{t_k}   & =&  \po \int_{t_{k-l}}^{t_k} e^{-\beta A_s(Z_s)}\dd s\pf^{-1}  \int_{t_{k-l}}^{t_k} \delta_{(Q_s,Z_s)} e^{-\beta A_s(Z_s)} \dd s
\end{eqnarray*}
for some small $l\in\N_*$, say $l=1$. In that case, in order to expect a long-time convergence of the bias, 
following classical stochastic algorithms, we would define the new bias as $A_{t_k} = A_{t_{k-1}} + \gamma_k f_k$ where $f_k$ is (a tensor approximation of) a minimizer over $H$ of $H\ni f\mapsto \J_{\nu_{t_k}}(A_{t_k}+f)$ 
and $(\gamma_k)_{k\in\N}$ is a positive sequence with $\gamma_k \rightarrow 0$ and $\sum_{l=1}^k \gamma_l \rightarrow \infty$ as $k\rightarrow \infty$.

\item Finally, a third way to deal with the memory management as time increases could be to 
use a stochastic gradient descent when solving the one-dimensional partial differential equation \eqref{eq:EDPrk}. In other words, 
when optimizing $r_{i}$ for some $1\leq i \leq d$, instead of computing averages over all steps $l\in\cco 1,t_k/\Delta t\ccf$, only use an approximation of $\nu_{t_k}$
by picking a random (and comparatively small) set of steps among $\cco 1,t_k/\Delta t\ccf$. Then only an estimation of the gradient of $H^1(\T) \ni r_i\mapsto \mathcal J_{t_k}(f+\bigotimes_{j=1}^d r_j)$ 
is computed, which is exactly the settings of the stochastic gradient descent.
\end{enumerate}

A second possible limitation is the following. Note that, as the number of reaction coordinates increases, 
we can expect that, at some point, the biasing scheme becomes unefficient. Indeed, by flattening the energy landscape, we replace the initial 
sampling problem (that was mainly restricted to low-energy regions, which form a low-dimensional manifold) by the sampling of the uniform measure on 
some hypercube, which is not so easy. In some sense, following the definitions of \cite{LelievreMetastable}, at some point, energy barriers
are replaced by entropic ones \nv{(which means that, in the exploration of the space, what takes time is not crossing high energy areas but visiting all areas in a relatively high dimensional space)}. Moreover, the variance of the estimator \eqref{eq:ImportanceSampling} increases due to the exponential weights. As a consequence, 
as $d$ increases, a partial biasing with $V_{bias,t} = \theta A_t\circ\xi$ for some $\theta\in(0,1)$ may be more appropriate than the full biasing (i.e. $\theta =1$). At the biased equilibrium, 
if $A_t = A$ is the true free energy, the marginal law of the reactions coordinates is thus $\mu_{(1-\theta)A,\beta}$, i.e. 
the temperature is increased. Then the choice of $\theta$ such that this measure satisfies a Poincar\'e inequality with minimal constant (which means the 
corresponding overdamped Langevin process mixes the fastest) may not be $\theta=1$.


\section{Numerical experiments}\label{SectionNumerique}

Let us fix some details and parameters that will hold for the different examples below. In this section, the modifications 
discussed in Section~\ref{Sec-Discuss-Real-Implement} are enforced. 

The one-dimensional functions $r_{n,j}$ are stored for all $n\in \mathbb{N}$ and $1\leq j \leq d$ on a discrete grid with $M=30$ points, 
so that the minimization of functions of the form $\bigotimes_{j=1}^d r_j \mapsto \J_{\nu_t}(f+\bigotimes_{j=1}^d r_j)$ 
is restricted to tensor products of one-dimensional continuous piecewise linear functions on this grid, and the Euler-Lagrange equations \eqref{eq:EDPrk} are replaced by $M\times M$ linear systems. 
This discretization replaces the regularization by a kernel $K$, which is no more necessary.  
The process \eqref{EqEABF} is discretized with a time-step $\delta t = 25.10^{-5}$, while the time integral in the empirical measure $\nu_t$ defined in \eqref{Eq-Empiric-Distrib} is discretized with a time-step 
$\Delta t = 20\delta t$. In other words, the reaction coordinates and the associated local mean forces are recorded in memory only every 
20 steps of the Euler scheme. Moreover, $N$ independent replicas of the processes are run in parallel and the occupation measure used to defined the bias is 
$\tilde \nu_{N,t}$ given by \eqref{Eq-nutildeN}. The update times $t_k$ of the bias are fixed at $t_k = kT$, with $T$ a multiple of $\Delta t$ 
and the number of tensor terms $g_{n}$ added at each update time is fixed with value $m$.

\subsection{A low dimensional  example}\label{subsec:toymodel}

We start to test the method on a toy model, with $N=30$ replicas, a bias update period of $T=100\Delta t$, a regularization parameter  $\lambda = 10^{-5}$, and $m=8$ tensor products added at each update. 
The reaction coordinates are Euclidean coordinates, more precisely $\xi(x)=(x_1,x_2)$, so that we don't introduce any additional extended coordinate. 
%


The dimensions are $D=3$, $d=2$, particles start at $(0,0,0)$ and
\begin{eqnarray*}
V(x_1,x_2,x_3)& =& - \sin(3x_1)\sin(x_2)\cos(x_3-1) +  \cos(3x_2+2)(0.5+\cos(x_3-2))\\
& & \ +\  2\sin(2x_1+0.5)\cos(x_3) - 5\cos(x_1)\cos(x_2)\cos(x_3+1)\,.
\end{eqnarray*}
This potential has the following properties: it is not a tensor 
product and yields a metastable process but, since $V(x_1,x_2,x_3) = \psi(x_1,x_2) \cos(x_3 + \varphi(x_1,x_2))$ for some functions $\psi$ and $\varphi$, there is no metastability in the orthogonal space for fixed $x_1,x_2$.

The results are given in Figures \ref{Figure_energiebet1} and \ref{Figure_histobet1} (for $\beta=1$) and \ref{Figure_energiebet5} and
\ref{Figure_histobet5} (for $\beta=5$). In both cases, the theoretical free energy is successfully computed and the histograms of the reaction coordinates is eventually flat. This is a bit slower with the inverse temperature $\beta=5$, since the initial metastability is very strong. As can be seen in 
Figure \ref{Figure_histobet5}, at that temperature and in the same times, a non-biased process is stuck in its initial well.

\begin{figure}[!h]
\begin{center}
\includegraphics[scale=0.25]{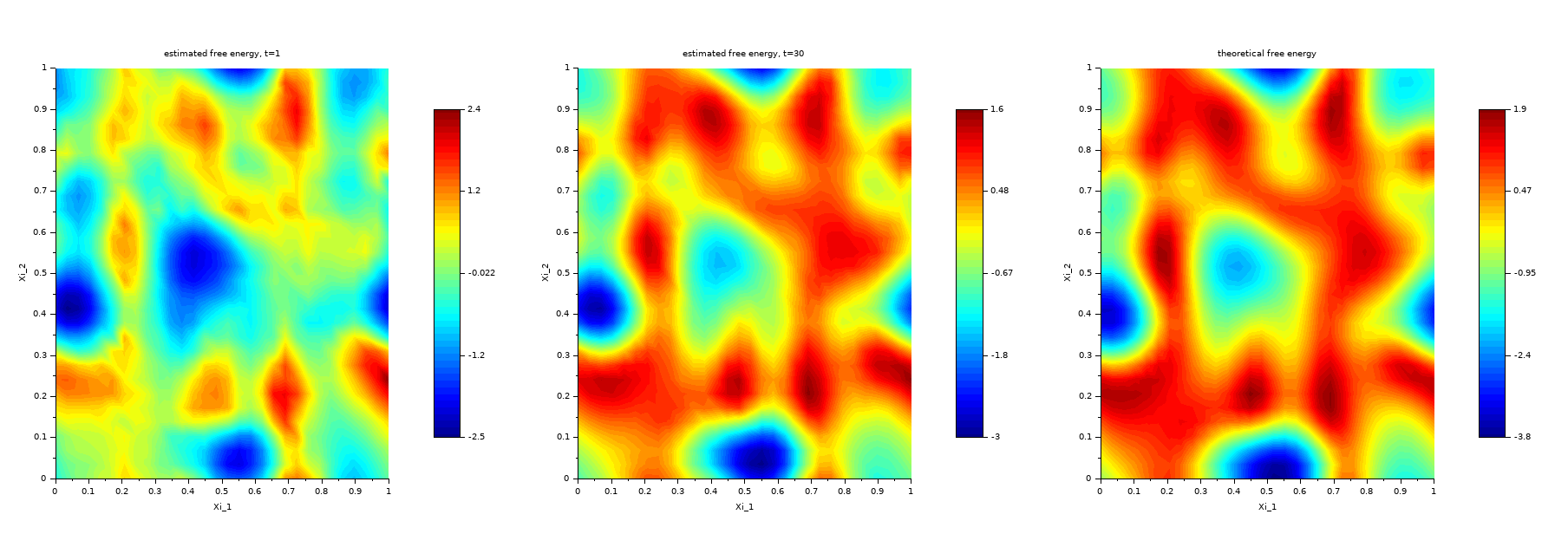}
\caption{For $\beta=1$, left and middle: estimated free energy respectively at $t=1$ and $t=30$. Right: theoretical free energy.}
\label{Figure_energiebet1}
\end{center}
\end{figure}

\begin{figure}[!h]
\begin{center}
\includegraphics[scale=0.25]{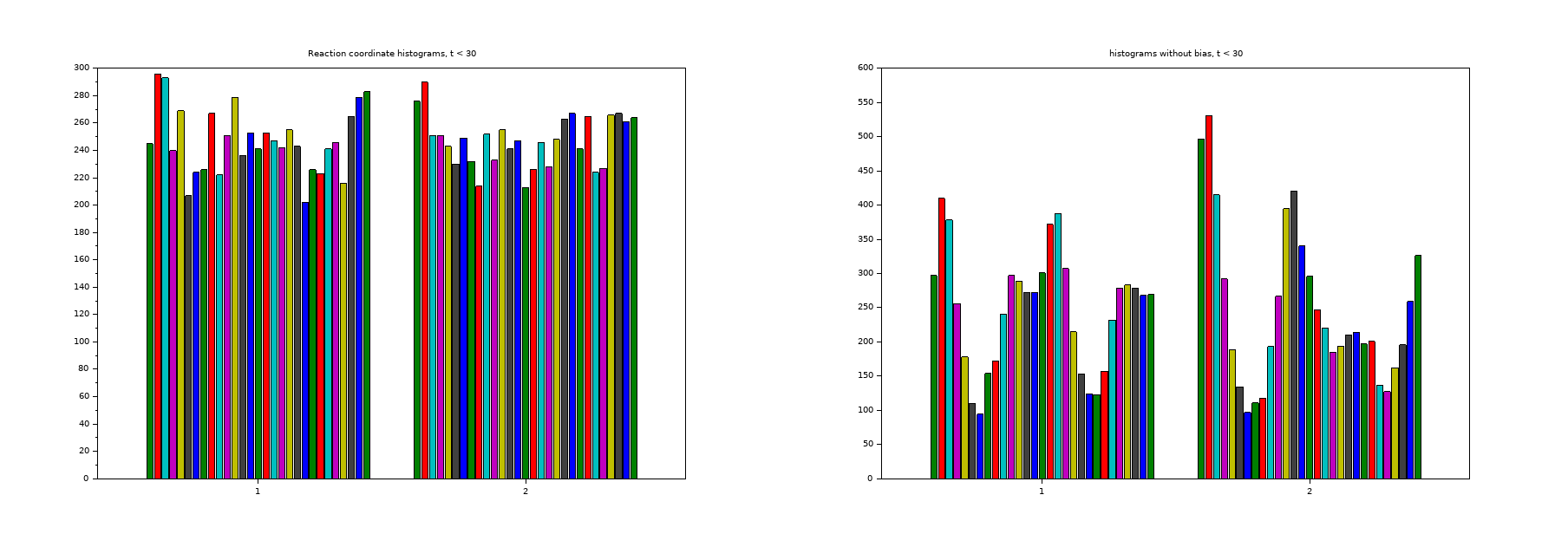}
\caption{For $\beta=1$, cumulated histograms of the reaction coordinates at $t=30$ respectively for the TABF algorithm (left) and a non-biased process (right).}
\label{Figure_histobet1}
\end{center}
\end{figure}

\begin{figure}[!h]
\begin{center}
\includegraphics[scale=0.25]{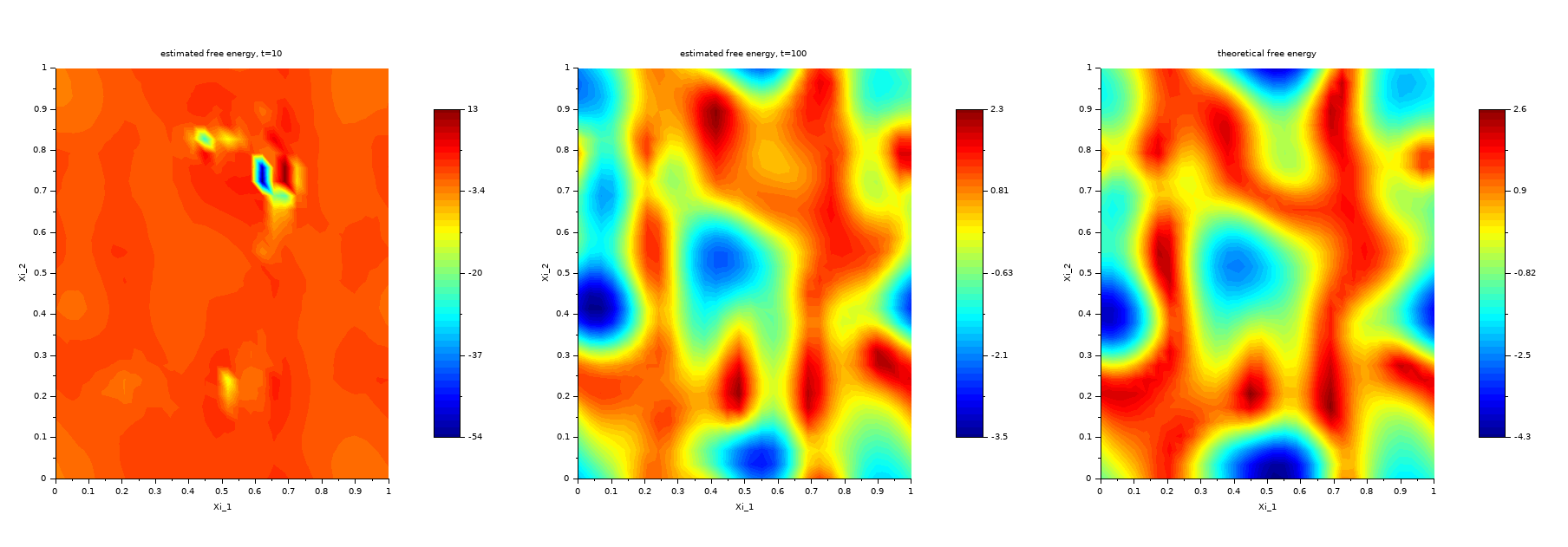}
\caption{For $\beta=5$, left and middle: estimated free energy respectively at $t=10$ and $t=100$. Right: theoretical free energy.}
\label{Figure_energiebet5}
\end{center}
\end{figure}

\begin{figure}[!h]
\begin{center}
\includegraphics[scale=0.25]{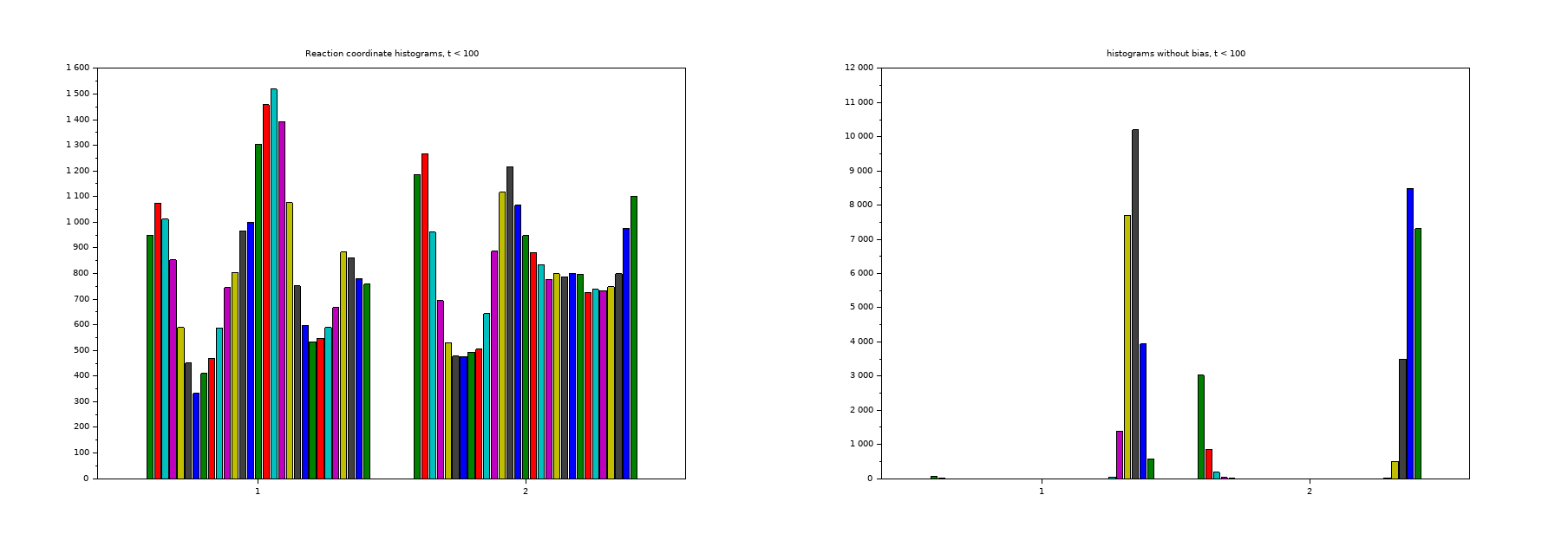}
\caption{For $\beta=5$, cumulated histograms of the reaction coordinates at $t=100$ respectively for the TABF algorithm (left) and a non-biased process (right).}
\label{Figure_histobet5}
\end{center}
\end{figure}

\subsection{Polymer ring in solvent}

We now consider a system inspired from \cite{LelievreAlrachid}. The system is constituted of two types 
of particles, solvent particles, and polymer particles. The polymer particles interact through a potential made precise below to form a ring.  The reaction coordinates are 
the bond lengths between consecutive polymer particles. This gives a large dimensional problem, for which the total dimension and the 
number of reaction coordinates are easily prescribed, and moreover where the reaction coordinates should exhibit some correlations (if it wasn't the case, 
then the TABF algorithm would not give better results than GABF  \cite{ChipotEGABF}).

In a two-dimensional periodic box, we consider $D/2=100$ particles among which $d$ (labeled from 1 to $d$) form a polymer and the others are solvent particles. 
The length of the box is $L=\sqrt{D/2}$, to ensure a concentration independent from $D$. Each pair of particles that involves at least one solvent particle 
interacts through the purely repulsive WCA pair potential, which is the  Lennard-Jones potential truncated at its minimum, namely
\begin{eqnarray*}
V_{WCA}(r)  &= & \varepsilon\1_{r\leqslant r_0}\po 1 + \po \frac{\sigma}{r}\pf^{12} - \po \frac{\sigma}{r}\pf^{6} \pf 
\end{eqnarray*}
where $r$ denotes the distance between the two particles, $\varepsilon=1$,  $\sigma=0.5$    and $r_0= 2^{1/6} \sigma$.
Each pair of consecutive particles in the polymer ring (where the $d^{th}$ and first particles are considered to be consecutive, closing the loop) interacts through  a double well potential 
\begin{eqnarray*}
V_{DW}(r) & = & h \po 1 - \frac{(2r-2r_1-\omega)^2}{\omega^2}\pf^2,
\end{eqnarray*}
where $r_1=r_0$, $\omega=1$ and $h=3$.   The minimum of $V_{DW}$ is attained at $r=r_1$ (compact state) and $r=r_1+w$ (stretched state). Finally, each triplet
of consecutive particles in the polymer  also interacts through the angle $\theta$ they form with the potential
\begin{eqnarray*}
V_{A}(\theta) & = & \frac{1}2 \po \cos(\theta) - \cos(\theta_d)\pf^2
\end{eqnarray*}
with an equilibrium angle $\theta_d = \pi(1-2/d)$ that ensures that the total angular potential is minimized when the polymer particles form a regular polygon.

There are $d$ reaction coordinates, which are the distances between two consecutive polymer particles. Following 
Section \ref{Sec-EABF}, the interaction between an extended reaction coordinate $z$ and the corresponding distance $r$  in the system is given via the extended potential
\begin{eqnarray*}
V_{E}(z,r) & =& \frac{1}{2\delta} \po z - \frac{r-r_1}{w}\pf^2 
\end{eqnarray*}
for   $\delta = 0.01$. The scaling ensures that the minimum of $V_{E}(z,r) + V_{DW}(r) $ is attained at $z=(r-r_1)/w \in \{0,1\}$. Moreover,
in line with Section \ref{Sec-Reflection}, the extended variable is confined in $[\xi_{min},\xi_{max}]^d$ by orthogonal reflection at the boundary,
with $\xi_{min} = -0.2$ and $\xi_{max} = 1.2$. 
%
\begin{figure}[!h]
\begin{center}
\includegraphics[scale=0.3]{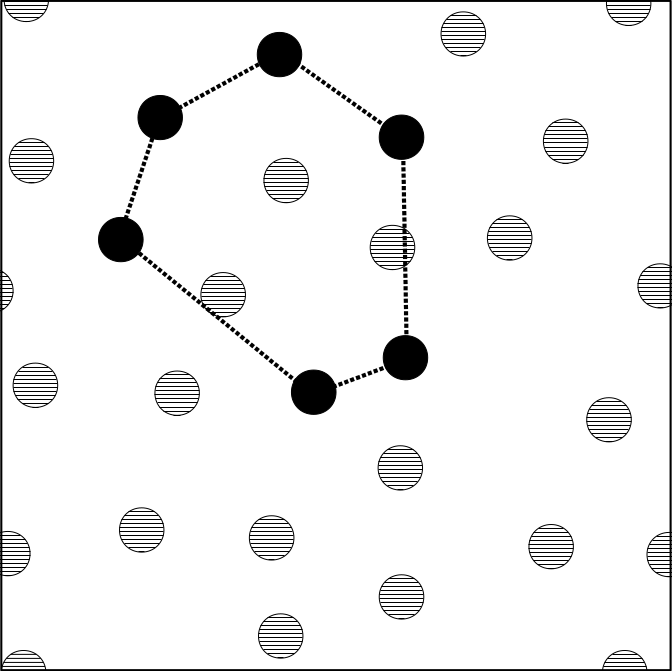}
\caption{The slow motions of the system are the transitions of each bond between two consecutive particles of the polymer from its compact state  to its stretched state.}
\label{Figure_Polymer}
\end{center}
\end{figure}

The total energy of the (extended) system is thus, for $(q,z)\in(L\mathbb T)^{D}\times [\xi_{min},\xi_{max}]^d$,
\begin{eqnarray*}
V(q,z) & =& \sum_{i=d+1}^{D/2} \sum_{j<i} V_{WCA}\po |q_i-q_j|\pf +  \sum_{i=1}^d   V_{E}\po z_i,|\tilde q_{i+1}-q_i|\pf   \\
& & + \sum_{i=1}^d V_{DW}\po |q_i - \tilde q_{i+1}|\pf  +  \sum_{i=1}^{d-1} V_A\po \arccos \po\frac{\tilde q_{i+1}-q_i}{|\tilde q_{i+1}-q_i|}\cdot \frac{\tilde q_{i+2}-\tilde q_{i-1}}{|\tilde q_{i+2}-\tilde q_{i+1}|} \pf \pf\,,
\end{eqnarray*}
where $\tilde q_i = q_i$ for all $i\in\cco 1,d\ccf$ and $\tilde q_{d+j} = q_j$ for $j=1,2$. Initially, the polymer is in a compact state, i.e. 
the distances between two consecutive of its particles are at distance $r_1$, the  angles are  $\theta_d$ and all the extended variables  $(z_i)_{i\in\cco 1,d\ccf}$ are at 0. 
For this model, we use the variant described in point 5 of Section \ref{Sec-Variations} namely, following the GABF algorithm, we keep in memory one dimensional free energies on a grid and we use the tensor approximation as a correction of this initial guess. There are $N=50$ replicas, the update period, regularization parameter, and inverse temperature are respectively $T=10^4\Delta t$, $\lambda = 0.05$ and $\beta=1$, and at each update, $m=4d$ tensor products are added.

 The free energy is expected to be close to a sum of one-dimensional  double well potentials, with minima attained at points close to 0 and 1. Nevertheless 
 the angular force should favor configurations where the consecutive distances in the polymer are close. This fact cannot be grasped by the GABF algorithm alone, 
 for which reaction coordinates are treated independently one from the others.
 
  The results are presented in Figure \ref{Figure_Polymer3} for $d=3$ and Figures \ref{Figure_Polymer5} and \ref{Figure_Polymer5_histo} for $d=5$.  
  In Figure~\ref{Figure_Polymer3}, we see that indeed the one-dimensional free energies recovered by the GABF algorithm have two wells approximately at 0 and 1, and that the  
  non-independent part of the free energy has the following effect: when $z_3=0$,  the well $(0,0)$ is favored, when $z_3=1$ the same goes for $(1,1)$, and when $z_3$ is intermediate
  the landscape is flatter and the two wells $\{z_1=z_2=x\}$ with $x\in\{0,1\}$ are favored with respect to the wells $(0,1)$ and $(1,0)$. The result is similar in Figure \ref{Figure_Polymer5}, even 
  though the quality of the estimation is lower for $z_3=z_4=z_5=0.5$, which is to be expected as this lies in a very low probability area (since $0.5$ is the saddle point of the two well potential). This
  shows that the TABF algorithm is able to recover non-trivial correlations between reaction coordinates.
  
  \nv{For $d=5$, the   free energy is eventually approximated with $140$ tensor terms ($7$ updates, adding $20$ terms each), which means $5\times 140 = 700$ one-dimensional functions have been stored, each represented by $M=30$ numbers. This is orders of magnitude below the cost $30^5=2,43 \times 10^{7}$ required to store a $5$-dimensional function on a grid of the same precision. Moreover, since the total number of time steps of the simulation is of order $10^6$, most of the points of the $5$-dimensional grid  have never been visited during the whole simulation, so we would'nt have any estimation of the free energy with a classical ABF algorithm.}
 

\begin{figure}[!h]
\begin{center}
\includegraphics[scale=0.33]{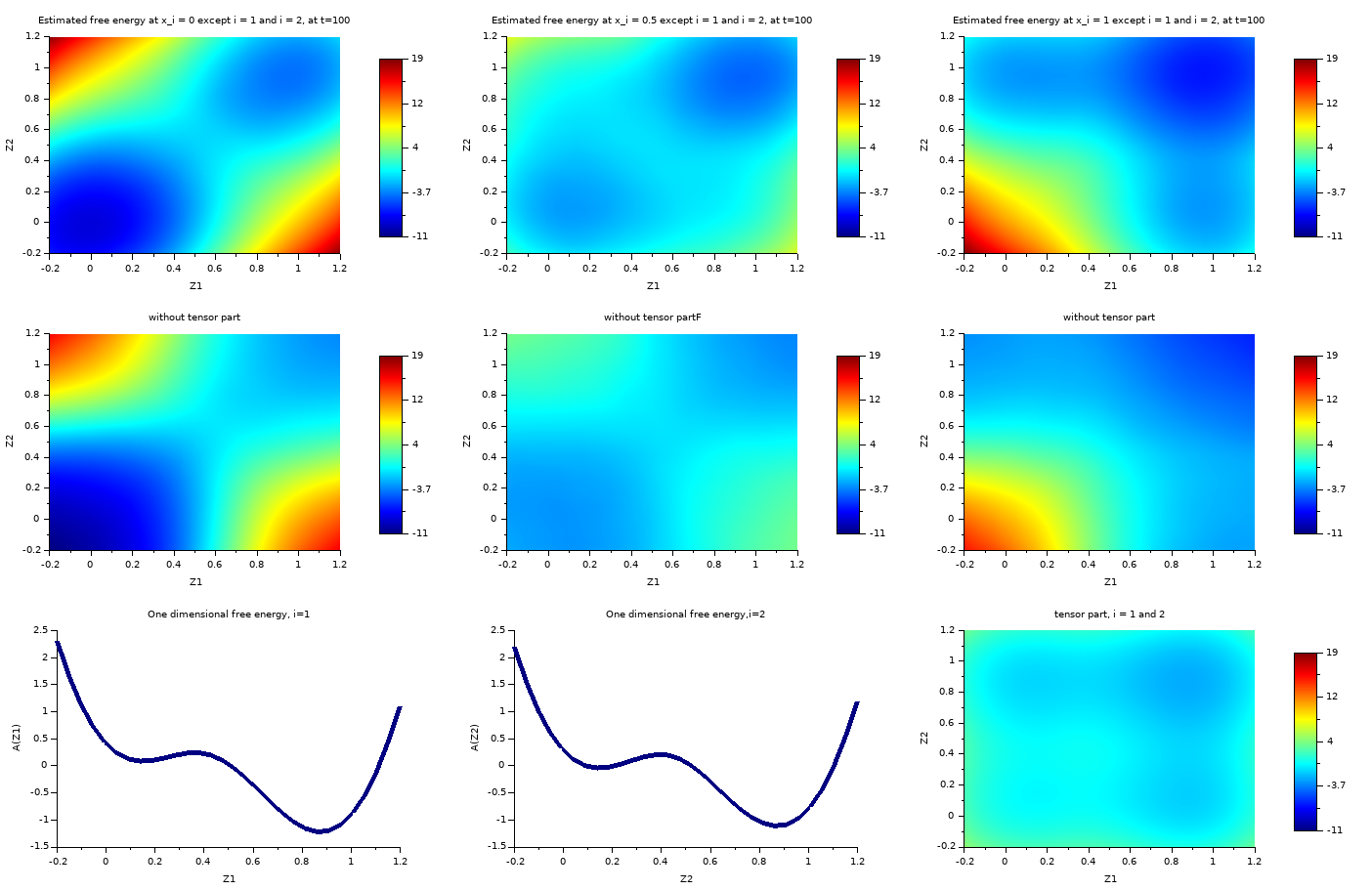}
\caption{For $d=3$,  up: estimated free energy as a function of $(z_1,z_2)$ when $z_3$ is, respectively, 0 (left) 0.5 (middle) and 1 (right). Middle: 
idem but without the independent, one-dimensional parts given by the GABF algorithm. Bottom: one-dimensional potential given by the GABF algorithm for $z_1$ (left) $z_2$ (middle) and their sum (right). }
\label{Figure_Polymer3}
\end{center}
\end{figure}


\begin{figure}[!h]
\begin{center}
\includegraphics[scale=0.3]{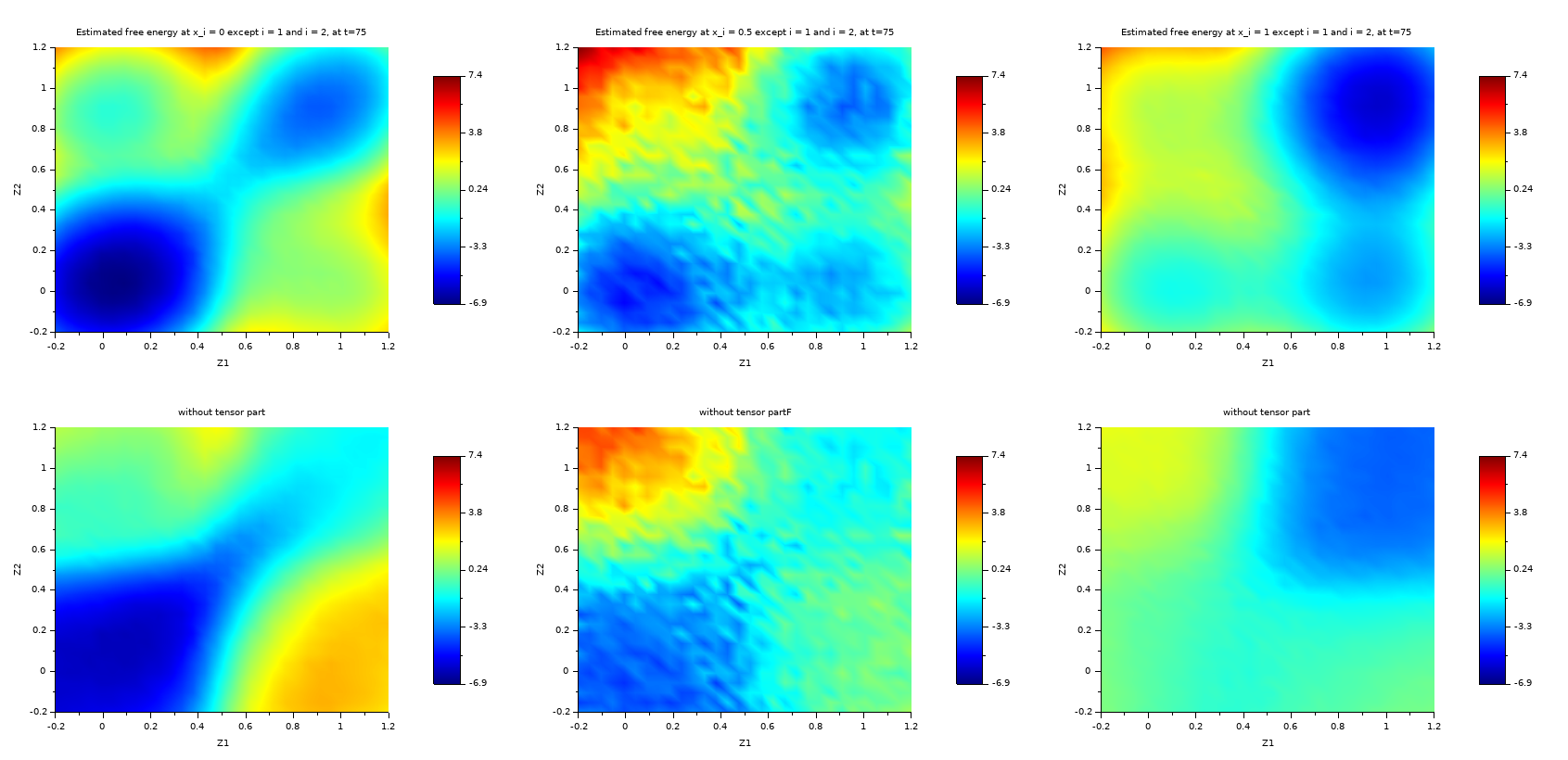}
\caption{For $d=5$, up: estimated free energy as a function of $(z_1,z_2)$ when $z_3$, $z_4$, $z_5$ are, respectively, $(0,0,0)$ (left) $(0.5,0.5,0.5)$ (middle) and $(1,1,1)$ (right). Bottom: idem but without the independent, one-dimensional parts given by the GABF algorithm. }
\label{Figure_Polymer5}
\end{center}
\end{figure}

\begin{figure}[!h]
\begin{center}
\includegraphics[scale=0.3]{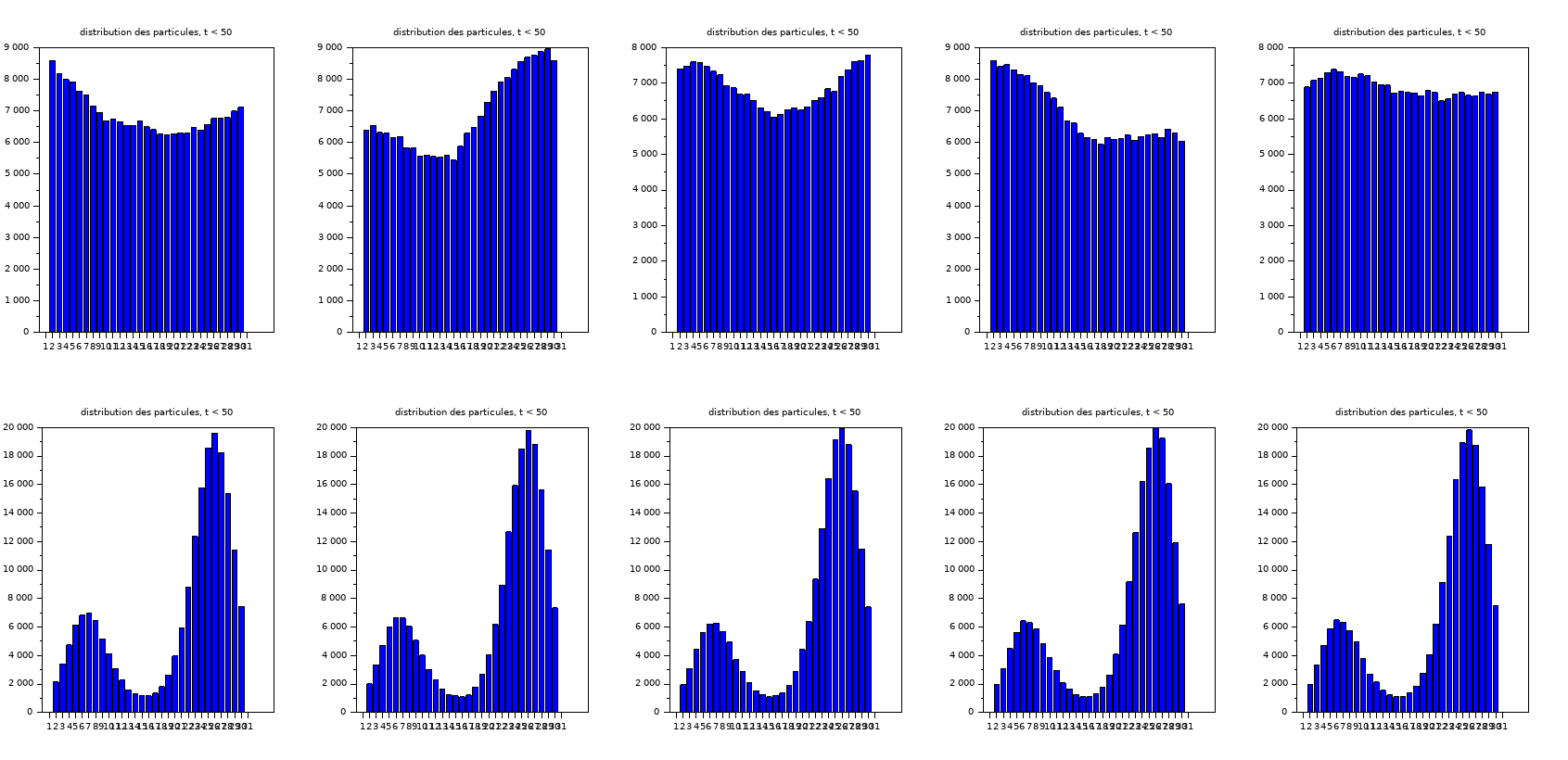}
\caption{For $d=5$, cumulative one-dimensional histograms of the five reaction coordinates at $t=50$ for the TABF algorithm (up) and for a non biased process (bottom).}
\label{Figure_Polymer5_histo}
\end{center}
\end{figure}

\section*{Acknowledgements}

This work was supported by the European Research Council under the European Union's Seventh Framework Programme (FP/2007-2013) / ERC Grant Agreement number 614492 and under the European Union's Horizon 2020 Research and Innovation Programme, ERC Grant Agreement number 810367, project EMC2. It was also supported by the ANR JCJC project COMODO (ANR-19-CE46-0002) and the ANR Project EFI (ANR-17-CE40-0030) of the French National Research Agency. The authors would like to thank the associate editor and the referees for their work.



\bibliographystyle{plain}
\bibliography{biblioTABF}

\end{document}